\documentclass[12pt,reqno]{article}
\usepackage[body={6.5in,9.0in},top=1.0in,left=1.0in,centering]{geometry}                % See geometry.pdf to learn the layout options. There are lots.
\usepackage{graphicx}
 \usepackage{amsmath,amsthm,amssymb,amsfonts,multirow,setspace,epsfig,latexsym}
\usepackage[letterpaper=true,colorlinks=true,urlcolor=blue,linkcolor=blue,citecolor=blue,pdfstartview=FitH]{hyperref}
\usepackage{pstricks}
\usepackage{epstopdf}
  \usepackage{dsfont}

 \usepackage{blindtext}


 \usepackage[sort,comma]{natbib}
\allowdisplaybreaks

\makeatletter \@addtoreset{equation}{section}

\newtheorem{thm}{Theorem}[section]
\newtheorem{coro}[thm]{Corollary}
\newtheorem{prop}[thm]{Proposition}
\newtheorem{lem}[thm]{Lemma}
\theoremstyle{definition}
\newtheorem{defn}[thm]{Definition}
\newtheorem{rem}[thm]{Remark}
\newtheorem{example}[thm]{Example}
\newtheorem{assumption}[thm]{Assumption}

\def\U{\ensuremath {\mathcal U}}
\def\S{\ensuremath {\mathcal S}}
\def\G{\ensuremath {\mathcal G}}
\def\B{\ensuremath {\mathcal B}}
\def\R{\ensuremath {\mathbb R}}
\def\C{\ensuremath {\mathbb C}}
\def\one{{\hbox{1{\kern -0.35em}1}}}

\newcommand{\e}{\varepsilon}
\newcommand{\al}{\alpha}
\newcommand{\ga}{\gamma}

\newcommand{\N}{\ensuremath{\mathbb N}}
\renewcommand{\P}{\ensuremath{\mathbb P}}

\newcommand{\E}{\ensuremath{\mathbb E}}
\newcommand{\EE}{\ensuremath{\mathcal E}}

\newcommand{\F}{\ensuremath{\mathcal F}}

\newcommand{\M}{\ensuremath{\mathcal M}}

\newcommand{\A}{\ensuremath{\mathcal A}}
\renewcommand{\L}{\ensuremath{\mathcal{L}}}

\newcommand{\op}{\ensuremath{\mathcal{L}}}

\newcommand{\wdt}{\widetilde}

\newcommand{\supkd}{\sup_{(k-1)\delta \le t \le k\delta}}
\newcommand{\kdel}{(k-1)\delta}
\newcommand{\set}[1]{\left\{#1\right\}}
\newcommand{\abs}[1]{\left\vert#1\right\vert}

\renewcommand{\d}{\mathrm{d}}
\newcommand{\cd}{(\cdot)}
\renewcommand{\(}{\left(}
\renewcommand{\)}{\right)}
\newcommand{\lan}{\langle}
\newcommand{\ran}{\rangle}

\def\tr{\ensuremath {\mbox{tr}}}
\title{Almost Sure and  Moment Exponential  Stability of Regime-Switching Jump Diffusions}
\author{Zhen Chao\thanks{Department of Mathematics, Shanghai Key Laboratory of Pure Mathematics and Mathematical Practice, East China Normal University, 500 Dongchuan Road, Shanghai, 200241, China, Email: {\sf zhenchao1120@163.com}.},   \quad
  Kai Wang\thanks{Department of Applied Mathematics, Anhui University of Finance and Economics, Bengbu 233030, China, Email: {\sf wangkai050318@163.com}.}, 
\quad  Chao Zhu\thanks{Department of Mathematical Sciences, University of Wisconsin-Milwaukee, Milwaukee, WI 53201, Email: {\sf zhu@uwm.edu}.},\quad and 
\quad  Yanling Zhu\thanks{School of International Trade and  Economics, University of International Business and Economics, Beijing 100029, China, Email: {\sf zhuyanling99@126.com}.}}

\begin{document}
\maketitle

\begin{abstract}

This work is  devoted to  almost sure and moment exponential stability of regime-switching jump diffusions. The Lyapunov function method is used to derive sufficient conditions for stabilities for general nonlinear systems;  which further helps to derive easily verifiable conditions for linear systems.  For one-dimensional linear regime-switching jump diffusions,  necessary and sufficient conditions for almost sure and $p$th moment exponential stabilities are presented. Several examples are provided for illustration.


\end{abstract}

\setlength{\baselineskip}{0.20in}
\section{Introduction}\label{sect-introduction}
Applications of  stochastic analysis % and stochastic control
 have
emerged in various areas such as financial
engineering,  wireless communications, mathematical biology, and risk management.
One of the salient features of   such systems is % that
  the coexistence of  and correlation between  continuous  dynamics and
discrete events. Often, the trajectories of these systems are not continuous: there is day-to-day jitter that causes minor fluctuations  as well as big jumps
 caused by rare events arising from,
e.g., epidemics, earthquakes, tsunamis, or terrorist atrocities.
On the other hand, the systems often display qualitative changes. For example, as demonstrated in \cite{BW87}, the volatility and  the expected rate of return  of an asset are markedly
different
in the bull and bear markets.
 Regime-switching diffusion  with L\'evy type jumps
naturally captures  these inherent features of these systems:
the L\'evy   jumps are well-known to incorporate both small and big jumps (\cite{APPLEBAUM,Cont-Tankov})
while the regime switching mechanisms provide the qualitative changes of the environment (\cite{MaoY,YZ-10}).
In other words,  regime-switching diffusion  with L\'evy jumps
provides a  uniform and realistic
 yet mathematically tractable
 platform in modeling a wide range of applications.
Consequently  increasing
attention has been drawn to the study of
 regime-switching jump diffusions in recent years. Some recent work in this vein can be found in \cite{Xi-09,Zong-14,Yin-Xi-10,ShaoX-14,ZhuYB-15} and the references therein.

 Regime-switching jump diffusion  processes can be viewed as jump diffusion processes in random environments, in  which the evolution of the random  environments is modeled by  a continuous-time Markov chain or more generally, a continuous-state-dependent switching process with a discrete state space. Seemingly similar to the usual jump diffusion processes, the behaviors of
  % the underlying systems are
  regime-switching jump diffusion  processes can be markedly  different. For example, \cite[Section 5.6]{YinZ13}  illustrates that two stable diffusion processes can be combined via a continuous-time Markov chain to produce an unstable regime-switching diffusion process.  See also  \cite{Costa-13} for similar observations.

 This paper aims to investigate almost sure and moment exponential stability for regime-switching  diffusions  with L\'evy type jumps.  % The underlying process is given by the stochastic differential equation
This is %a continuum of our efforts of studying asymptotic properties of regime-switching diffusion processes in \cite{KZY07,YZ-10}, and
 motivated by the recent advances in the investigations of stability of regime-switching jump diffusions in \cite{Yin-Xi-10,Zong-14} and the references therein. In  \cite{Yin-Xi-10,Zong-14}, the L\'evy measure $\nu$ on some measure space $(U,\mathfrak U)$ is assumed to be a finite measure with $\nu(U) < \infty$. Consequently,  in these models,  the jump mechanism is modeled by compound Poisson processes and there are  finitely many jumps in any finite time interval. 
In contrast, % with  the aforementioned references, which treat stability and stabilization of  regime-switching jump diffusions with  compound Poisson type jumps, 
% this paper is focused on the stability of  regime-switching jump diffusions with L\'evy type jumps. 
in our formulation, the L\'evy measure $\nu$ on $(\R^{n}-\{0\}, \B(\R^{n}-\{0\}))$ merely satisfies  $\int_{\R^{n}-\{0\}} (1\wedge |z|^{2})\nu(\d z)< \infty$ and hence it is not necessarily finite. %;  see Section \ref{sect-formulation} for details. 
This formulation allows the possibility of infinite number of ``small jumps'' in a finite time interval.  % In finance literature, this phenomenon is called infinite activity. 
% At one hand,  regime-switching jump diffusion processes with L\'evy type jumps are more general as well as more challenging to analyze than those with  compound Poisson type jumps. For compound Poisson type jumps, In contrast, % Indeed, even though technical aspects of the proofs of several results in \cite{Yin-Xi-10,Zong-14} and this paper are similar,   the consideration of L\'evy type jumps does introduce additional subtleties  and hence requires extra care in the analysis; see, for example, the proofs of Theorem \ref{thm-as-exp-stable} and Propositions \ref{prop-exp-stab-creterion-b} and \ref{prop-moment-exp-stab-1d}.  
Indeed, %the consideration of L\'evy type jumps
% our  work is   motivated by
% in view of %the recent applications    such as
such   ``infinite activity models'' are studied in the finance literature, such as the variance gamma model in \cite{Seneta-04} and the normal inverse Gaussian model in \cite{Barn-98}. See also the recent paper  \cite{Barn-13} for     energy spot price modeling using L\'evy processes. 

Our focus of this paper is to study  almost sure and   moment exponential stabilities of the equilibrium point  $x=0$ of   regime-switching jump diffusion processes. To this end, we first observe the ``nonzero'' property, which asserts that almost all sample paths of all solutions to  \eqref{sw-jump-diffusion} starting from a nonzero initial condition will never reach the origin with probability one. This phenomenon was first established for diffusion processes in \cite{khasminskii-12} and later extended to regime-switching diffusions in \cite{MaoY,YZ-10} under the usual Lipschitz and linear growth conditions. For processes with L\'evy type jumps, additional assumptions are needed to handle the jumps to obtain the ``nonzero'' property. For instance,   \cite{ApplS-09} and \cite{Wee-99} contain different sufficient conditions. The  differences are essentially on the assumptions concerning the jumps. Here we propose a different sufficient condition than those in  \cite{ApplS-09,Wee-99} for the ``nonzero'' property for regime-switching jump diffusion.  We show in Lemma \ref{lem-0-inaccessible} that  the ``nonzero'' property holds under the usual Lipschitz and linear growth conditions on the coefficients of \eqref{sw-jump-diffusion}  together with  Assumption \ref{assump-jump}. Note that it is quite easy to verify Assumption \ref{assump-jump} in many practical situations; see, for example, the discussions in Remark \ref{rem-concerning-assumption-jump}. 
%The conditions imposed in  Lemma \ref{lem-0-inaccessible} below is similar to those in Lemma 3.3 of \cite{ApplS-09}.  A set of related but different conditions can be found in \cite{Wee-99}.

%The main results of this paper include
With the  ``nonzero'' property at our hands, we proceed to obtain sufficient conditions for almost sure and $p$th moment exponential stabilities of the equilibrium point of  nonlinear   regime-switching jump diffusions.  %For general nonlinear systems,  %system \eqref{sw-jump-diffusion}  in Section \ref{sect:nonlinear}. These
Similar to the related results in \cite{ApplS-09} for jump diffusions, these sufficient conditions for stability are expressed in terms of the existence of appropriate Lyapunov functions. The details are spelled out in Theorems \ref{thm-as-exp-stable} and \ref{thm-moment-stab}, and Corollary \ref{coro-as-exp-stable}. Also, as observed  in \cite{Costa-13,YinZ13,YZ-10} for   regime-switching diffusions, our results demonstrate that the switching mechanism can contribute to the  stabilization or destabilization of jump diffusion processes.  
 % Because of the presence of regime switching, some additional analyses are needed.  
 % As with the cases of diffusions in \cite{khasminskii-12} and regime-switching diffusions in \cite{MaoY}, 
 Next we  show in Theorem \ref{thm-moment-as-stable} that $p$th ($p\ge 2$) moment exponential stability implies almost sure exponential stability for regime-switching jump diffusions   under a certain integrability  condition on the jump term.  Such a result has been established for  diffusions in \cite{khasminskii-12}, jump diffusions in \cite{ApplS-09}, and regime-switching diffusions in \cite{MaoY}. In addition,    we   derive  a sufficient condition for $p$th moment exponential stability using $M$-matrices in Theorem \ref{thm-moment-stab-m-matrix}.

The aforementioned general results  are then applied to treat linear %systems.  Next we focus on almost sure and moment exponential stability for linear 
regime-switching jump diffusions. For one-dimensional systems, we obtain necessary and sufficient conditions for almost sure and $p$th moment exponential stabilities in Propositions \ref{prop-as-exp-stable-1d} and \ref{prop-moment-exp-stab-1d}, respectively. For the multidimensional system, we present   verifiable sufficient conditions %  in terms of the eigenvalues of certain matrices 
 for almost sure and moment exponential  stability in Propositions \ref{prop-exp-stab-creterion-b}, \ref{cor-p-stab}, and  \ref{prop-as-moment-stab-M-matrix}. To illustrate the results, we also study several examples in Section \ref{sect:examples}.

The remainder of the paper is organized as follows. After a brief introduction to regime-switching jump diffusion processes in Section \ref{sect-formulation}, we proceed to deriving sufficient conditions for almost sure and $p$th moment exponential stabilities of the equilibrium point of the nonlinear system \eqref{sw-jump-diffusion}  in Section \ref{sect:nonlinear}.
 % These conditions are expressed in terms of the existence of appropriate Lyapunov functions; from which we also obtain a sufficient condition using $M$-matrices.
 %We also show that under certain conditions, $p$th moment stability implies almost sure exponential stability.
Section \ref{sect-linear} treats stability of the equilibrium point of linear systems.
%  For the one-dimensional system \eqref{eq-SDE-1}, we obtain necessary and sufficient conditions for almost sure and $p$th moment exponential stabilities. For the multidimensional system  \eqref{linearsystem}, we present easily verifiable sufficient conditions in terms of eigenvalues of certain matrices for stability.
Finally we conclude the paper with conclusions and remarks in Section \ref{sect-conclusion}.

To facilitate the
 presentation, we introduce some notation
 that will be used often in later sections.
   Throughout the paper, we use
   $x'$  to denote the transpose of $x$, and
 $x'y$  or $x\cdot y$ interchangeably to denote the inner product of
   the vectors $x$ and $y$. If
     %$x = (x_{1}, \dots, x_{n})'\in \R^{n}$, then $|x| : = \sqrt{x_{1}^{2}+ \dots +  x_{n}^{2}}$.  For $A \in \R^{n \times n}$,
   $A$ is a vector or matrix, then $|A| : = \sqrt{\tr(AA')}$, $\| A\|:= \sup\{|Ax|: x\in \R^{n}, |x| =1\}$, and $A\gg 0$ means that every element of $A$ is positive.
   For a square matrix $A $,  $\rho(A)$ is the spectral radius of $A$. Moreover if $A$ is a symmetric square matrix, then $\lambda_{\max}(A) $ and $\lambda_{\min}(A)$ denote the maximum and minimum eigenvalues of $A$, respectively.
     For sufficiently smooth function $\phi: \R^n \to \R$, $D_{x_i} \phi= \frac{\partial \phi}{\partial x_i}$,
    $D_{x_ix_j} \phi= \frac{\partial^2 \phi}{\partial x_i\partial x_j}$, and we denote by $D\phi   =(D_{x_1}\phi, \dots, D_{x_n}\phi)'\in \R^{n}$ and $D^2\phi =(D_{x_ix_j}\phi) \in \R^{n\times n}$ the gradient and  Hessian of $\phi$, respectively.
    For $k \in \mathbb N$, $C^{k}(\R^{n})$ is the collection of functions $f: \R^{n }\mapsto \R$
    with continuous partial derivatives up to the $k$th order while  $C^{k}_{c} (\R^{n})$ denotes the space of $C^{k}$ functions with compact support.
 If $B $ is a set, we use $B^o$ and $I_B$ to denote the interior and indicator function of $B$, respectively. Throughout the paper, we adopt the conventions that $\sup \emptyset =-\infty$
 and $\inf \emptyset = + \infty$.

\section{Formulation}\label{sect-formulation}

Let  $(\Omega, \F, \set{\F_{t}}_{t\ge 0}, \P)$ be  a filtered probability space   satisfying the usual condition on
which is defined  an $n$-dimensional standard $\F_t$-adapted Brownian motion   $W\cd$.
Let $\set{\psi(\cdot) }$ be an $\F_t$-adapted   L\'evy process with L\'evy measure $\nu(\cdot)$.
Denote  by  $N(\cdot,\cdot)$
     the corresponding $\F_t$-adapted Poisson random measure defined on $\R_+ \times \R^n_0$: $$ N(t,U):= \sum_{0 < s \le t}I_{U}( \Delta \psi_{s} )= \sum_{0< s \le t} I_{U}(\psi(s)- \psi(s-)),$$
     where $t \ge 0$ and $U $ is a Borel subset of $\R^{n}_{0}=\R^{n}-\set{0}$.  % with compensator $\tilde N$,
     % and intensity measure $\nu\cd$,
   %   where $\R_+=[0,\infty)$, $\R^n_0= \R^n -\set{0}$.
 The compensator $\tilde N$ of $N$ is given by \begin{equation}
\label{eq-N-tilde}
\tilde N(\d t,\d z):= N(\d t,\d z) - \nu(\d z)\d t.
\end{equation}  Assume that $W\cd$ and $N(\cdot,\cdot)$ are independent and that
$\nu\cd$ is a  % $\sigma$-finite measure on $\R_0^n$.
 L\'evy measure satisfying
 \begin{equation}
\label{levy-measure} \int_{\R^n_0} (1 \wedge \abs{z}^2)\nu(\d z) < \infty,
\end{equation} where  $a_1\wedge a_2 =\min \{a_1, a_2\}$ for  $a_1,a_2 \in \R $.

 We consider a stochastic differential equation with regime-switching together with L\'evy-type jumps of the form
\begin{equation}\label{sw-jump-diffusion} \begin{aligned}
 \d X(t) = & \, b(X(t ),\al(t))\d t + \sigma(X(t ),\al(t ))\d W(t)  \\
    & \  + \int_{\R ^n_0} \gamma(X(t-),\al(t-),z)\tilde N(\d t,\d z),
  \ \ t \ge   0,
  \end{aligned} \end{equation}
  with initial conditions
\begin{equation}\label{swjd-initial}
 X(0)=x_0 \in \R ^n, \ \ \al(0) = \al_0\in \M,  \end{equation}
where $b(\cdot,\cdot) : \R ^n \times \M \mapsto \R ^{n}$,
$\sigma (\cdot,\cdot): \R ^n\times \M \mapsto \R ^{n\times n}$, and $ \gamma (\cdot,\cdot,\cdot): \R ^n \times \M \times \R ^n_0
\mapsto  \R ^{n}$ are measurable functions, and
   $\al\cd$ is a  %switching process continuous-time Markov chain
switching component   with  a finite state space $\M:=\set{1, \dots, m}$ and infinitesimal generator $Q= (q_{ij}(x))\in \R^{m\times m}$. That is, $\al\cd$ satisfies
\begin{equation}\label{Q-gen}\P\set{\al(t+ \delta)=j| X(t) =x,
\al(t)=i,\al(s),s\le t}=\begin{cases}q_{ij}(x)
\delta + o(\delta),&\hbox{ if }\ j\not= i,\\
1+ q_{ii}(x)\delta + o(\delta),  &\hbox{ if }\ j=i,
\end{cases}
 \end{equation} as $\delta \downarrow 0$,  where $q_{ij}(x)\ge 0$ for $i,j\in \M$
with $j\not= i$ and $ q_{ii}(x)=-\sum_{j\not= i}q_{ij}(x) <0$
 for each $i\in \M$.

   The evolution of the discrete component
  $\al\cd$ in \eqref{sw-jump-diffusion}
can be represented by a stochastic integral with respect to a Poisson random
measure; see, for example, \cite{Skorohod-89}. In fact,  for
$x\in \R^n$ and $i,j
\in \M$ with $j\not=i$, let $\Delta_{ij}(x)$ be
the consecutive  left-closed, right-open
intervals of
the half real line $\R_{+} : = [0,\infty)$, each having length $q_{ij}(x)$. In case $q_{ij}(x) =0$, we set $\Delta_{ij}(x) =\emptyset$. Define a function
$h: \R^n \times \M \times \R \mapsto \R$ by
\begin{equation}\label{h-def} h(x,i,z)=\sum^m_{j=1} (j-i)
I_{\{z\in\Delta_{ij}(x)\}}.
\end{equation}
Then the evolution of  the switching process \eqref{Q-gen} can be represented by the  stochastic differential equation
\begin{equation}\label{eq:jump-mg}
\d\al(t) = \int_{\R_{+}} h(X(t-),\al(t-),z) {N_1}(\d t,\d z),\end{equation}
where ${N_{1}}(\d t,\d z)$ is a Poisson random measure  (corresponding to a random point process $\mathfrak p\cd$) with intensity $\d t
\times \lambda(\d z)$, and $\lambda\cd$
is the Lebesgue measure on $\R$.  Denote the compensated Poisson random measure of $N_1\cd$
 by $\tilde N_{1}(\d t, \d z): = N_{1}(\d t, \d z)-\d t\times  \lambda (\d z)$. Throughout this paper, we assume that the L\'evy process $\psi\cd$, the random point process $\mathfrak p\cd$,
 and the Brownian motion
$W\cd$ are independent.

%To study stability of the system given by \eqref{sw-jump-diffusion} and \eqref{eq:jump-mg}, we assume
We make  the following assumptions    throughout the paper:
\begin{assumption}\label{assump-0} Assume
\begin{equation}\label{0-equilibrium}
b(0, i) = \sigma(0,i) = \int_{\R_{0}^{n} } \gamma(0, i, z) \nu(\d z) = 0\; \text{ for all } i\in \M.
\end{equation}
\end{assumption}
 \begin{assumption}\label{assump-ito}
For some positive constant $\kappa$,   we have
\begin{align}\label{ito-condition}
&  \begin{aligned}
& \abs{b(x,i)- b(y,i)}^2 + \abs{\sigma(x,i)- \sigma(y,i)}^2
 \\ & \qquad  \qquad \qquad
 + \int_{\R ^n_0} \abs{\gamma(x,i,z)- \gamma(y,i,z)}^2\nu(\d z) \le \kappa \abs{x-y}^2,
 \end{aligned}\\
 \label{eq2-ito-condition}
&\qquad
%\abs{ b(x,i)}^2 + \abs{\sigma(x,i)}^2 +
\int_{\R^n_0} \bigl[\abs{\ga(x,i,z)}^2 + |x\cdot \ga(x,i,z)| \bigr] \nu(\d z)
 \le  \kappa |x|^{2}   %  \kappa(1+\abs{x}^2),\\
 \end{align}
 for all $x,y \in \R ^n$  and $i\in \M=\set{1, \dots, m}$, and that
 \begin{equation}
\label{eq-q-bdd}
\sup \set{q_{ij}(x): x \in \R ^{n}, i\not= j \in \M} \le \kappa < \infty.
\end{equation}
\end{assumption}

Under Assumptions \ref{assump-0} and \ref{assump-ito}, $X(t) \equiv 0$ is an equilibrium point of \eqref{sw-jump-diffusion}. Moreover,
in view of \cite{ZhuYB-15},  for each initial condition $(x_0,\al_0) \in \R ^n \times \M$, the system represented by \eqref{sw-jump-diffusion} and \eqref{Q-gen} (or equivalently, \eqref{sw-jump-diffusion} and \eqref{eq:jump-mg})
has a unique strong solution  $(X\cd,\al\cd)=(X^{x_0,\al_0}\cd,\al^{x_0,\al_0}\cd)$; the solution does not explode in finite time  with probability one. In addition, the generalized It\^o lemma
  reads
\begin{equation*}\label{eq:ito}   f(X(t),\al(t)) -f(x_0,\al_0)
 = \int^t_0 {\mathcal L}f(X(s-),\al(s-)) \d s
+ M_1^f(t) + M_2^f(t) +M_3^f(t),  \end{equation*}  for $f\in C^{2}_{c} (\R^{n}\times\M)$,  where
$\mathcal L$ is the operator associated with the
process $(X,\al)$ defined by:
\begin{equation*}\label{op-defn}\begin{aligned}
\op f(x,i) =  &\, D f(x,i)\cdot b(x,i) + \frac{1}{2}\tr((\sigma\sigma')(x,i)D^2 f(x,i)) + \sum_{j\in \M} q_{ij}(x) [f(x,j)-f(x,i)]
\\  &\ + \int_{\R^n_{0}} \![f(x+ \gamma(x,i,z),i)-f(x,i)-
      D f(x,i)\cdot\gamma(x,i,z) ] \nu(\d z), \, (x,i)\in \R^{d}\times \M, \end{aligned}\end{equation*}
 and
\begin{align*}  M_1^f(t) &  = \int^t_0  D
f(X(s-),\al(s-)) \cdot \sigma(X(s-),\al(s-))
\d W(s), \\
 M_2^f(t) &  =  \int_0^t \int_{\R_+} \big[ f( X(s-), \al(s-)+
h(X(s-),\al(s-),z))
-f(X(s-),\al(s-))\big] \tilde N_{1}(\d s, \d z), \\
  M_3^f(t) & = \int_0^t \int_{\R^n_0} \!\left[f(X(s-) + \gamma(X(s-),\al(s-),z),\al(s-))- f(X(s-),\al(s-))\right]\tilde N(\d s, \d z).
 \end{align*}

  %  Then % under Assumption \ref{standing-assumption},
 %   by virtue of \cite{ZhuYB-15}, for any initial condition $( x, i ) \in \R^{n}\times \M$,  the system given by \eqref{linearsystem} and  \eqref{eq:jump-mg} has a unique strong solution $(X(\cdot),\al(\cdot)) = (X(\cdot;x,i),\al(\cdot;x,i))$. In addition, we have $\E_{x,i} [\sup_{t\in [0,T]}\abs{X(t)}^p] \le K <\infty,   (x,i) \in \R^n \times \M,$ for any $T > 0$ and $p \in (0, 2]$.

Similar to the terminologies in \cite{khasminskii-12}, we have
\begin{defn}
The equilibrium point of \eqref{sw-jump-diffusion} is said to be
\begin{enumerate}
% \item[{\rm (i)}] {\em stable in probability}, if for any $ \varepsilon  >0$ and any $i\in\M$,
% \begin{displaymath}
%  \lim_{x\to 0} \P\set{\sup_{t\ge 0} |X(t; x,i)| > \varepsilon } =0;
% \end{displaymath} and $x=0$ is said to be {\em unstable in probability} if it is not stable in probability.

%\item [{\rm (ii)}] {\em asymptotically stable in probability}, if
% it is stable in probability and satisfies\begin{displaymath} \lim_{x\to 0} \P\set{\lim_{t\to \infty}X(t; x,i)=0 }=1, \text{ for any} i\in\M.
% \end{displaymath}
  \item[(i)] {\em almost surely exponentially stable} if  there exists a $\delta > 0$ independent of $(x_0,\alpha_0)\in \R^{n}_{0}  \times \M$ such that
  $$\limsup_{t\to\infty} \frac{1}{t}\log |X^{x_0,\alpha_0}(t)| \le -\delta\, \text{ a.s.} $$
  \item[(ii)] {\em exponentially stable in the $p$th moment} if there exists a $\delta > 0$ independent of $(x_0,\alpha_0)\in  \R^{n}_{0} \times \M$ such that
  $$\limsup_{t\to\infty} \frac{1}{t}\log \E[|X^{x_0,\alpha_0}(t)|^{p}] \le -\delta. $$
\end{enumerate}
\end{defn}

To study    stability of the equilibrium point of \eqref{sw-jump-diffusion}, we first present the following ``nonzero'' property, which asserts that almost all sample paths of all solutions to  \eqref{sw-jump-diffusion} starting from a nonzero initial condition will never reach the origin. This phenomenon was first established for diffusion processes in \cite{khasminskii-12} and later extended to regime-switching diffusions in \cite{MaoY,YZ-10} under fairly general conditions. For processes with L\'evy type jumps, additional assumptions are needed to handle the jumps.
%The conditions imposed in  Lemma \ref{lem-0-inaccessible} below is similar to those in Lemma 3.3 of \cite{ApplS-09}.  A set of related but different conditions can be found in \cite{Wee-99}.

\begin{assumption}\label{assump-jump}
Assume there exists a   constant  $\varrho >0$ such that 
 \begin{align}
\label{eq-no-jump-to-0}
%& \nu\set{z\in \R_{0}^{n}: \text{ there exists }(x,i) \in \R^{n}_{0}\times \M \text{ such that } %x\not= 0 \text{ and } x + \gamma(x,i,z) =0 } =0, 
|x+ \gamma(x,i,z)| \ge \varrho |x|, \text{ for all }(x,i) \in \R^{n}_{0}\times \M \text{ and }\nu\text{-almost all }z\in \R^{n}_{0}. 
\end{align}
\end{assumption} 
\begin{rem}\label{rem-concerning-assumption-jump} From  \cite{MaoY,YZ-10}, we know that under Assumptions \ref{assump-0} and \ref{assump-ito}, a regimes-switching diffusion without jumps cannot ``diffuse'' from a nonzero state to zero a.s. Assumption \ref{assump-jump} further prevents the process $X$ of \eqref{sw-jump-diffusion} jumps from a nonzero state to zero. 

Also a  sufficient condition for \eqref{eq-no-jump-to-0} is   \begin{displaymath}
 2x\cdot \gamma(x,i,z) + |\gamma(x,i,z)|^{2} \ge 0,
\end{displaymath} for $\nu$-almost all $z\in \R_{0}^{n}$ and all $(x,i) \in \R^{n}\times \M$.
Indeed, under such a condition, we have $|x + \gamma(x,i,z)|^{2 } = |x|^{2} + 2x\cdot \gamma(x,i,z) + |\gamma(x,i,z)|^{2} \ge |x|^{2} $ for $\nu$-almost all $z\in \R_{0}^{n}$. This, of course,  implies  \eqref{eq-no-jump-to-0}.  
\end{rem}
% {\blue We can probably use the argument in Fang and Zhang, Theorem C here.}

\begin{lem}\label{lem-0-inaccessible} Suppose  Assumptions \ref{assump-0}, \ref{assump-ito}, and  \ref{assump-jump} hold. 
 Then  % for any $x\neq y\in  \R^{n}$ and $i\in \M$, we have   \begin{equation} \label{eq-non-confluence property} \P\{X^{x,i}(t) \neq X^{y,i}(t), \text{ for all }t \ge 0\} =1.
% \end{equation} In particular, 
  for any $(x,i) \in \R^{n}_{0}\times \M$, we have
\begin{equation}\label{eq-path-0}
\P_{x,i} \{X(t) \neq 0 \text{ for all } t \ge 0 \} =1.
\end{equation}
\end{lem}

\begin{proof} Consider the function $V(x,i): = |x|^{-2}$ for $x\neq 0$ and $i\in \M$.  Direct calculations reveal that % \begin{align*}   
$D V(x,i)= -2|x|^{-4} x,    $  and  $  D^2 V(x,i)=-2
  |x|^{-4}I  +8  |x|^{ -6}   x x'. $ % \end{align*} 
  Next we prove that for all $x,y\in \R^{n}$ with $x\neq 0$ and $|x+y| \ge \varrho |x| $, we have
\begin{equation}
\label{eq-|x|-2-estimate}
V(x+y,i) - V(x,i) - DV(x,i)\cdot y = \frac{1}{|x+y|^{ 2}} - \frac{1}{|x|^{ 2}} + \frac{2  x' y}{|x|^{4}} \le K \frac{ |y|^{2}+ |x'y|}{|x|^{4}},
\end{equation} where $K$ is some positive constant. 
Let us    prove \eqref{eq-|x|-2-estimate} in several cases: 

{\bf Case 1: $x'y \ge 0$}. In this case, it is easy to verify that for any $\theta\in [0,1]$, we have $|x+ \theta y |^{2} = |x|^{2} + 2 \theta x'y + \theta^{2}|y|^{2} \ge |x|^{2}$. Therefore we can use  the Taylor expansion with integral reminder to compute  
\begin{align*}
    |x+y|^{-2} - |x|^{-2} + 2 |x|^{-4} x' y &  = \int_{0}^{1} \frac12 y\cdot D^{2} V(x+ \theta y) y \, \d \theta \\
    &   =   \int_{0}^{1} \biggl[- \frac{|y|^{2}}{|x+ \theta y|^{4}}+ 4\frac{y'(x+ \theta y)(x+ \theta y)' y}{|x+ \theta y |^{6}} \biggr]\d \theta\\
    & \le 4\int_{0}^{1} \frac{|y|^{2}}{|x+ \theta y|^{4}} \d \theta 
      \le 4\int_{0}^{1} \frac{|y|^{2}}{|x|^{4}} \d \theta =  \frac{4 |y|^{2}}{|x|^{4}}.
\end{align*}

{\bf Case 2: $x'y  <  0$ and $2x'y + |y|^{2} \ge 0$}.  In this case, we have $|x+y|^{2} = |x|^{2} + 2x'y+ |y|^{2} \ge |x|^{2}$ and hence $|x+y|^{-2} - |x|^{-2} \le 0$; which together with $x'y \le 0$ implies that 
\begin{displaymath}
|x+y|^{-2} - |x|^{-2} + 2 |x|^{-4} x' y \le 0.
\end{displaymath}  

{\bf Case 3: $x'y  <  0$ and $2x'y + |y|^{2} <  0$}.  In this case, we use the bound $|x+y| \ge \varrho |x| $ to compute % we have  $|x+y|^{2} < |x|^{2}$ and hence 
\begin{align*}
   |x+y|^{-2} - |x|^{-2} + 2 |x|^{-4} x' y &  =  \frac{1}{ |x+y|^{2}} - \frac{1}{|x|^{ 2}}  -\frac{|y|^2}{|x|^{4}}+ \frac{2 x'y+|y|^2}{|x|^{4}}     \\
    &  =\frac{-2 x'y}{|x|^{2} |x+y|^{2}}   -\frac{|y|^{2}}{|x|^{2} |x+y|^{2}} -\frac{|y|^2}{|x|^{4}} +  \frac{2 x'y+|y|^{2}}{|x|^{4}} \\
    & \le \frac{2 |x'y|}{\varrho^{2}|x|^{4}}. 
\end{align*}

Combining the three cases gives \eqref{eq-|x|-2-estimate}. 

Observe that \eqref{eq-no-jump-to-0} of Assumption \ref{assump-jump} implies that if $x\neq 0$, then $x+ c(x,i,z) \neq 0$ for $\nu$-almost all $z \in \R_{0}^{n}$. Therefore    we use Assumptions  \ref{assump-0} and \ref{assump-ito} and  \eqref{eq-|x|-2-estimate} to compute 
\begin{align}\label{eq-op-calculation} \nonumber \op V(x,i) & =  -2 \abs{x}^{-4}x \cdot b(x,i) 
   +  \frac{1}{ 2}\tr\left[ \sigma \sigma'(x,i)  |x|^{-6} \(-2 \abs{x}^2
 I + 8  x x'\)\right]\\
 \nonumber & \quad + \int_{\R^{n}_{0}}\big[ |x + \gamma(x,i,z)|^{-2} - |x|^{-2}  +   2 |x|^{-4} x \cdot \gamma(x,i,z)  \big]\nu(\d z)\\
\nonumber  & \le 2 \kappa |x|^{-2} + 4 |\sigma(x,i)|^{2} |x|^{-4} + K |x|^{-4} \int_{\R_{0}^{n}}\bigl[ |\gamma(x,i,z)|^{2}+ |x\cdot\gamma(x,i,z)|\bigr]\nu(\d z)\\
 & \le K |x|^{-2} = K V(x,i), \end{align} 
 where $K$ is a positive constant. 
 
%The rest of the proof is similar to the proof of Lemma 3.2 of \cite{ApplS-09}. Define for $\e > 0$ and $R > 0$, $\tau_{\e}: = \inf\{ t\ge 0: |X(t)| \le \e\}$ and $\tau_{R}: = \inf\{ t\ge 0: |X(t)| \ge R\}$. 
%Using this estimate, we can apply It\^o's formula to
Now consider  the process $  (X,\al)$ with initial condition $(X(0),\al(0)) = (x,i) \in \R_{0}^{n}\times \M$. Define for $0< \e < |x| < R$, $\tau_{\e}: = \inf\{ t\ge 0: |X(t) | \le \e\}$ and $\tau_{R} : = \inf\{ t\ge 0: |X(t)| \ge R\}$. Then   \eqref{eq-op-calculation} allows us to derive  
\begin{align*}
\E_{x,i}& [e^{-K  (t\wedge \tau_{\e} \wedge \tau_{R})}V(X(t\wedge \tau_{\e} \wedge \tau_{R}),\al(t\wedge \tau_{\e} \wedge \tau_{R}))]     
\\ & =  V(x,i) + \E_{x,i} \biggl[ \int_{0}^{t\wedge \tau_{\e} \wedge \tau_{R}} e^{-K  s} (-K   + \op) V(X(s),\al(s))\d s\biggr]    \\
    &   \le V(x,i) = |x|^{-2}, \ \ \text{ for all }t\ge 0.
\end{align*} Note that on the set $\{ \tau_{\e} < t \wedge \tau_{R}\}$, we have $V(X(t\wedge \tau_{\e} \wedge \tau_{R}),\al(t\wedge \tau_{\e} \wedge \tau_{R})) = |X(t\wedge \tau_{\e} \wedge \tau_{R})|^{-2} \ge \e^{-2}.$  Thus  it follows that 
\begin{displaymath}
e^{-K  t} \e^{-2} \P_{x,i}\{ \tau_{\e} < t \wedge \tau_{R}\} \le \E_{x,i}  [e^{-K  (t\wedge \tau_{\e} \wedge \tau_{R})}V(X(t\wedge \tau_{\e} \wedge \tau_{R}),\al(t\wedge \tau_{\e} \wedge \tau_{R}))]  \le |x|^{-2}.
\end{displaymath} It is well known that under Assumptions  \ref{assump-0} and \ref{assump-ito}, the process $X $ has no finite explosion time   and hence $\tau_{R} \to \infty$ a.s. as $R \to \infty$. Therefore for any $t> 0$, we have $\P_{x,i}\{\tau_{\e} < t \} \le e^{K  t} \e^{2}   |x|^{-2}  $. Passing to the limit as  $\e \downarrow 0$, we  obtain  $\P_{x,i}\{ \tau_{0} < t\} =0$ for any $t> 0$, where $\tau_{0}: = \inf\{t\ge 0: X(t) =0\}$. This gives \eqref{eq-path-0} and hence completes the proof. 
\end{proof}

\section{Stability of Nonlinear Systems: General Results}\label{sect:nonlinear}
This section is devoted to establishing sufficient conditions in terms of the existence of appropriate  Lyapunov functions  for stability of the equilibrium point of the system \eqref{sw-jump-diffusion}. Section \ref{subsect:as-stab} considers almost surely exponential stability while Section \ref{sect:pth-moment-stab} studies $p$th moment exponential stability and demonstrates that $p$th moment exponential stability implies almost surely exponential stability   under certain conditions.   Finally we present a sufficient condition for stability using $M$-matrices in Section \ref{sect:M-matrix}.
 \subsection{Almost Sure Exponential  Stability}\label{subsect:as-stab}
% [Next, a result similar to Theorem 3.7 of  \cite{ApplS-09} or Theorem 3 of \cite{Wee-99}.]

\begin{thm}\label{thm-as-exp-stable} Suppose  Assumptions \ref{assump-0}, \ref{assump-ito}, and  \ref{assump-jump} hold.  Let $V:  \R^n\times\mathcal{M} \mapsto\R^+ $ be such that $V(\cdot, i)\in C^{2}(\R^{n})$ for each $i\in \M$. Suppose there exist    $p>0$, $c_1(i) >0$, $c_2(i) \in \R$, and nonnegative constants $c_3(i) $, $c_4(i) $ and $c_5(i) $
  such that for all $x\not=0$ and $i\in \M$,
\begin{itemize}
  \item [(i)] $c_1(i) |x|^{p}\leq V(x,i),$
 \item[(ii)] $\op V(x,i)\leq c_2(i) V(x,i),$
 \item[(iii)] $| DV(x,i)\cdot \sigma(x,i)|^2\geq c_3(i) V(x,i)^2$,
\item[(iv)] $\displaystyle\int_{\R_0^{n}}\biggl[\log\left(\frac{V(x+\gamma(x,i,z),i)}{V(x,i)}\right)-\frac{V(x+\gamma(x,i,z),i)}{V(x,i)}+1\biggr]\nu(\d z)\leq-c_4(i) $,
 \item[(v)] $\displaystyle\sum_{j\in \mathcal{M}}q_{ij}(x)\left(\log V(x,j)-\frac{V(x,j)}{V(x,i)}\right)\leq -c_5(i) .$
\end{itemize}
Then
 % $$\limsup_{t\rightarrow \infty}\frac{1}{t}\log |X(t)|\leq \frac{1}{p}\sum_{i=1}^m\pi_i[c_{2}(i)  - 0.5\, { c_{3}(i) }- c_{4}(i)  -c_{5}(i)]:=\delta \;\text{ a.s.}$$
 \begin{equation}
\label{eq1-X-as-exp-stab}
 \limsup_{t\rightarrow \infty}\frac{1}{t}\log |X(t)|\leq  \frac{1}{p} \max_{i\in \M}\bigl\{c_{2}(i)-  0.5 c_{3}(i) -c_{4}(i) -c_{5}(i)\bigr \} =:\delta \;\text{ a.s.}
\end{equation}
In particular,  if $\delta<0$ then the trivial solution of  \eqref{sw-jump-diffusion}   is a.s. exponentially stable. %  for all $x\in \R^d.$
\end{thm}

\begin{rem}
Conditions (iv) and (v) of Theorem \ref{thm-as-exp-stable} are natural because of the following observations. At one hand,   using the elementary inequality $\log y\leq y-1$ for $y>0$ we derive
 $$\displaystyle\int_{\R_0^{n}}\biggl[\log\left(\frac{V(x+\gamma(x,i,z),i)}{V(x,i)}\right)-\frac{V(x+\gamma(x,i,z),i)}{V(x,i)}+1\biggr]\nu(\d z)\leq 0;$$ this leads us to assume that the left-hand side of the above   equation is bounded above by a nonpositive constant $-c_4(i) $ in condition (iv); however, the constant may depend on $i\in\M$. On the other hand,   the   inequality $\log y\leq y-1$ for $y>0$ also leads to
\[\log V(x,j)-\frac{V(x,j)}{V(x,i)}\leq \log V(x,i)-1.\]
Then it follows that  for every  $ i\in \mathcal{M}$,
\[
\aligned
&  \sum_{j\in \mathcal{M}}q_{ij}(x)\left(\log V(x,j)-\frac{V(x,j)}{V(x,i)}\right)  \\ & \quad 
=\sum_{j\not=i}q_{ij}(x)\left(\log V(x,j)-\frac{V(x,j)}{V(x,i)}\right)+q_{ii}(x)(\log V(x,i)-1)\\
 &\quad \leq \sum_{j\not=i}q_{ij}(x)\left(\log V(x,i)-1\right)+q_{ii}(x)(\log V(x,i)-1)= 0.
\endaligned 
\] In view of this observation, a nonpositive constant $-c_{5}(i)$ in condition (v) is therefore reasonable; again, this constant may depend on $i\in \M$. In fact, the constants $c_{1}(i), \dots, c_{5}(i) $ in Conditions (i)--(iv) may all depend on $i\in \M$; this allows for some extra flexibility for the selection of the Lyapunov function $V$ and  more importantly, the sufficient condition for a.s. exponential stability in Theorem  \ref{thm-as-exp-stable} 
and Corollary \ref{coro-as-exp-stable}. 

It is also worth poiting out that  conditions (i)--(iv) are similar to those in Theorem 3.1 of \cite{ApplS-09}. Condition (v) is needed so that we can   control the fluctuations of $\frac1t \log|X(t)|$ due to the presence of regime switching. 
\end{rem}

The proof of Theorem \ref{thm-as-exp-stable} is a straightforward extension of that of Theorem 3.1 of  \cite{ApplS-09};   some additional care are needed due to the presence of  regime-switching. For completeness and also to preserve the flow of presentation, we  relegate the proof to  the Appendix \ref{sect-appendix}.

 \begin{coro}\label{coro-as-exp-stable}
 In addition to the conditions of Theorem \ref{thm-as-exp-stable}, suppose also that the discrete component $\al$  in  \eqref{sw-jump-diffusion}  and \eqref{eq:jump-mg} is an irreducible continuous-time Markov chain with  an invariant distribution $\pi = (\pi_{i}, i \in \M)$, then \eqref{eq1-X-as-exp-stab} can be strengthened to
 \begin{equation}
\label{eq2-X-as-exp-stab}
 \limsup_{t\rightarrow \infty}\frac{1}{t}\log |X(t)|\leq  \frac{1}{p}  \sum_{i\in \M }\pi_{i} [c_{2}(i)  - 0.5\, { c_{3}(i) }- c_{4}(i)  -c_{5}(i)]\; \text{ a.s.}
\end{equation}
 \end{coro}
\begin{proof}
 This follows from applying the ergodic theorem of continuous-time  Markov chain (see, for example, \cite[Theorem 3.8.1]{Norris-98}) to the right-hand side of \eqref{eq2-U-prelimit}:
   \begin{align*} \lim_{t\to\infty} &  \frac{1}{t}  \int_{0}^{t} \bigl[c_{2}(\al(s))- 0.5 c_{3}(\al(s)) -c_{4}(\al(s)) -c_{5}(\al(s))\bigr]\d s \\
& = \sum_{i=1}^m\pi_i[c_{2}(i)  - 0.5\, { c_{3}(i) }- c_{4}(i)  -c_{5}(i)]\; \text{ a.s.}
  \end{align*} Then \eqref{eq2-X-as-exp-stab} follows directly.
\end{proof}

\subsection{Exponential $p$th-Moment Stability}\label{sect:pth-moment-stab}

 \begin{thm}\label{thm-moment-stab} Suppose Assumptions \ref{assump-0} and \ref{assump-ito}.
 Let $p, c_{1}, c_{2}, c_{3}$ be positive constants. Assume that there exists a function $V:\R^{n}\times \M \mapsto \R^{+}$ such that $V(\cdot,i)\in C^{2} (\R^{n})$ for each $i\in \M$ satisfying
 \begin{itemize}
  \item[(i)] $c_{1} |x|^{p} \le V(x,i) \le c_{2} |x|^{p}$,
  \item[(ii)]  $\op V(x,i)\leq -c_3V(x,i),$
\end{itemize} for all $(x,i) \in \R^{n}\times \M$. Let $(X(0),\al(0)) = (x,i) \in\R^{n}\times \M$. 
Then we have \begin{itemize}
  \item [(a)] %\begin{displaymath}
$\E[|X(t)|^{p}] \le \frac{c_{2}}{c_{1}} |x|^{p} e^{-c_{3} t}. $
%\end{displaymath}  %where $x_{0} = X(0) \in \R^{n}$.
 In particular, the equilibrium point of  \eqref{sw-jump-diffusion}  is  exponentially stable in the $p$th moment with Lyapunov exponent less than or equal to $-c_{3}$.

  \item[(b)] Assume in addition that $p \in (0, 2]$. Then there exists an almost surely finite and positive random variable $\Xi$ such that \begin{equation}
\label{eq new-as-exp stable}
|X(t)|^{p} \le \frac{\Xi}{c_{1}} e^{-c_{3}t} \text{ for all }t \ge 0 \text{ a.s.} 
\end{equation} In particular, the equilibrium point of  \eqref{sw-jump-diffusion}  is almost sure exponentially stable with Lyapunov exponent less than or equal to $-\frac{c_{3}}{p}$.
 
\end{itemize} \end{thm}
\begin{proof}
 The proof of part (a) is very similar to that of Theorem 3.1 in \cite{Mao-99}; see also Theorem 4.1 in \cite{ApplS-09}. For brevity, we shall omit the details here. 
 
 %Let's sketch the proof of part (b) following the idea in the proof of \cite[Theorem 5.15]{khasminskii-12}.
 Part (b) is motivated by  \cite[Theorem 5.15]{khasminskii-12}. For any $t \ge 0$ and $(x,i) \in\R^{n}\times \M$, we consider the function $f(t,x,i): = e^{c_{3}t}V(x,i)$.  Condition (i) and Lemma 3.1 of \cite{ZhuYB-15} imply that $$\E[f(t,X(t),\al(t))] = \E[e^{c_{3}t} V(X(t),\al(t))]\le c_{2 }e^{c_{3} t } \E[ |X(t)|^{p}] < \infty.$$ On the other hand, thanks to It\^o's formula, we have for all $0 \le s < t$
 \begin{align*}
  f&(t,X(t),\al(t)) \\  & =  f(s,X(s),\al(s))+  \int_{s}^{t}e^{c_{3}r } (c_{3} + \op)V(X(r),\al(r))\d r  \\ & \quad+  \int_{s}^{t}e^{c_{3}r }D V(X(r),\al(r)) \cdot \sigma(X(r),\al(r)) \d W(r) \\ 
  & \quad + \int_{s}^{t}\int_{\R_+} e^{c_{3}r }\big[ V( X(r-), \al(r-)+ h(X(r-),\al(r-),z))-V(X(r-),\al(r-))\big] \tilde N_{1}(\d r, \d z)\\
  &\quad+  \int_{s}^{t}\int_{\R^{n}_{0}}e^{c_{3}r } [V(X(r-) + \gamma(X(r-),\al(r-),z), \al(r-)) - V(X(r-),\al(r-))] \wdt N(\d r, \d z)\\ 
  & \le  f(s,X(s),\al(s))+ \int_{s}^{t}e^{c_{3}r }D V(X(r),\al(r)) \cdot \sigma(X(r),\al(r)) \d W(r) \\ & \quad + \int_{s}^{t}\int_{\R_+} e^{c_{3}r }\big[ V( X(r-), \al(r-)+ h(X(r-),\al(r-),z))-V(X(r-),\al(r-))\big] \tilde N_{1}(\d r, \d z)\\ &\quad+  \int_{s}^{t}\int_{\R^{n}_{0}}e^{c_{3}r } \big[V(X(r-) + \gamma(X(r-),\al(r-),z), \al(r-)) - V(X(r-),\al(r-))\big] \wdt N(\d r, \d z),\end{align*} where we used condition (ii) to obtain the inequality. 
  Let $\tau_{n}: = \inf\{t\ge0: |X(t)| \ge n\} $. Then we have $\lim_{n \to \infty}\tau_{n}= \infty$ a.s.  and $\E[f(t\wedge\tau_{n}, X(t \wedge \tau_{n}), \al(t\wedge \tau_{n})) |\F_{s}] \le f(s\wedge\tau_{n}, X(s \wedge \tau_{n}), \al(s\wedge \tau_{n}))$     a.s.  Passing to the limit as $n\to \infty$, and noting that $f$ is positive,  we obtain from Fatou's lemma that $\E[ f(t,X(t),\al(t))|\F_{s}] \le f(s,X(s),\al(s))$ a.s. Therefore it follows that the process $\{f(t,X(t),\al(t)), t \ge 0 \}$ is a positive   supermartingale. The martingale convergence theorem (see, for example, Theorem 3.15 and Problem 3.16 in \cite{Karatzas-S}) then implies that  $f(t,X(t),\al(t))$ converges a.s. to a  finite limit as $t \to \infty$. Consequently there exists an a.s. finite and positive  random variable $\Xi$ such that $$\sup_{t \ge 0}\{  e^{c_{3}t} V( X(t),\al(t))\} = \sup_{t \ge 0} \{ f(t, X(t),\al(t))\}  \le \Xi < \infty,  \text{ a.s.}$$ Furthermore, it follows from condition (i) that $|X(t)|^{p} \le \frac{1}{c_{1}} V(X(t),\al(t))$. Putting this observation into the above displayed equation yields \eqref{eq new-as-exp stable}.
\end{proof}

\begin{thm}\label{thm-moment-as-stable} Let Assumptions \ref{assump-0} and \ref{assump-ito} hold.
Suppose  the equilibrium point of \eqref{sw-jump-diffusion} is $p$th moment exponentially stable for some $p \ge 2$
 % \footnote{Here the key is Kunita's first inequality, which requires $p \ge 2$. What happens when $0 < p < 2$? }
 and that for some positive constant $\hat \kappa$, we have
\begin{equation}
\label{eq-pth-moment-as-stable}
\int_{\R^{n}_{0}} \abs{\gamma(x,i,z)}^{p}\nu(dz) \le \hat \kappa |x|^{p}, \quad (x,i) \in \R^{n} \times \M.
\end{equation} Then  the equilibrium point of \eqref{sw-jump-diffusion} is almost surely exponentially stable.
\end{thm}

The proof of Theorem \ref{thm-moment-as-stable} is very similar to the proofs of Theorem 4.2 of \cite{ApplS-09} and Theorem 5.9 of 
  \cite{MaoY} and  is  deferred to   Appendix \ref{sect-appendix}. Note that Theorem 4.4  in \cite{ApplS-09} requires a condition (Assumption 4.1) similar to \eqref{eq-pth-moment-as-stable} to hold for all $q \in [2,  p]$. Here we observe that it is enough to have  \eqref{eq-pth-moment-as-stable}   for a single $p$, as long as $p \ge 2$. Also notice that  in the special case when $p=2 $, then  \eqref{eq-pth-moment-as-stable} is already contained in Assumption \ref{assump-ito}. 

\subsection{Criteria for Stability Using $M$-Matrices}\label{sect:M-matrix}
In this subsection, we assume that in \eqref{Q-gen}, $Q(x) = Q\in \R^{m\times m}$, a constant matrix. Consequently, the switching component $\al(\cdot)$ in \eqref{sw-jump-diffusion}  is a continuous-time Markov chain. Let us also assume   \begin{assumption}\label{assump-m-matrix} There exist a positive number $p> 0$ and
a positive definite matrix $G\in \S^{n\times n}$ such that for all $x\neq 0$ and $i \in \M$, we have  \begin{align}
\label{eq1-moment-stable-M}
& \lan Gx, b(x,i) \ran + \frac{1}{2} \lan \sigma(x,i), G \sigma(x,i) \ran \le \varrho_{i} \lan x, Gx\ran,
\\ \label{eq2-moment-stable-M}
& (\lan x, G\sigma(x,i) \ran)^{2} \begin{cases}
\le \delta_{i}   (\lan x, G x \ran)^{2},     & \text{ if } p \ge 2, \\
  \ge     \delta_{i}   (\lan x, G x \ran)^{2},     & \text{ if } 0 < p  < 2,
\end{cases}
\intertext{and}\label{eq3-moment-stable-M}
& \int_{\R_{0}^{n}}\biggl[ \biggl( \frac{\lan x+\gamma(x,i,z), G (x+ \gamma(x,i,z))\ran}{\lan x, Gx\ran}\biggr)^{\frac p2} - 1 -\frac{p\lan \gamma(x,i,z), Gx\ran}{\lan x, Gx\ran} \biggr] \nu(\d z) \le \lambda_{i},
\end{align} where $\varrho_{i}, \delta_{i},$ and $\lambda_{i}$, $i\in \M$ are constants.   \end{assumption}
Corresponding to the infinitesimal generator  $Q$ of \eqref{Q-gen} and $p> 0$,   $G\in \S^{n\times n}$ in Assumption \ref{assump-m-matrix}, we define an $m\times m$ matrix
\begin{equation}
\label{eq-matrix-A(pG)}
\A: =\A(p, G) = \textrm{diag}(\theta_{1}, \dots, \theta_{m}) - Q,
\end{equation} where $\theta_{i}: = -[p\varrho_{i}  + p (p-2) \delta_{i} + \lambda_{i} ]  $,    $ i =1, \dots, m. $
\begin{thm}\label{thm-moment-stab-m-matrix}
Suppose  Assumptions \ref{assump-0}, \ref{assump-ito}, \ref{assump-jump}, and  \ref{assump-m-matrix} hold and that the matrix $\A$ defined in \eqref{eq-matrix-A(pG)} is a nonsingular $M$-matrix,   then    the equilibrium point of  \eqref{sw-jump-diffusion}  is  exponentially stable in the $p$th moment. In addition, if  % in Assumption \ref{assump-m-matrix}, 
either $p\in (0, 2]$ or $p > 2$ with \eqref{eq-pth-moment-as-stable} valid, then  then    the equilibrium point of  \eqref{sw-jump-diffusion}  is a.s. exponentially stable. 
\end{thm}
Recall from \cite{MaoY} that a square matrix $A= (a_{ij})\in \R^{n\times n}$ is a nonsingular $M$-matrix if $A$ can be expressed in the form $A= s I- G$ with some $G\ge0$ and $s> \rho(G)$, where $I$ is the identity matrix and $\rho(G)$ denotes the spectral radius of $G$. 
\begin{proof} Since $\A$ of \eqref{eq-matrix-A(pG)} is a nonsingular $M$-matrix, by Theorem 2.10 of \cite{MaoY}, there exists a vector $(\beta_{1}, \dots, \beta_{m})' \gg 0$ such that $(\bar \beta_{1}, \dots, \bar \beta_{m})' : = \A (\beta_{1}, \dots, \beta_{m})' \gg 0 $. Componentwise, we can write
\begin{displaymath}
\bar \beta_{i} : = \theta_{i} \beta_{i} - \sum_{j\in \M} q_{ij} \beta_{j} > 0, \quad i \in \M.
\end{displaymath}
Define $V(x,i) : = \beta_{i} (x' Gx)^{\frac p2}$ for $(x,i)\in \R^{n}\times \M$. Then condition (i) of Theorem \ref{thm-moment-stab} is trivially satisfied. Moreover, we can use Assumption \ref{assump-m-matrix} to compute
\begin{align*}
 \mathcal L V(x,i) & = \beta_{i} p (x' Gx)^{\frac p2 - 1} \lan b(x,i), Gx\ran  + \sum_{j=1}^{m} q_{ij} \beta_{j}  (x' Gx)^{\frac p2}
  \\ & \quad+ \frac12 \beta_{i} p (x' Gx)^{\frac p2-2} \tr\big(\sigma(x,i) \sigma'(x,i) [(p-2) Gxx'G + x' Gx G] \big)\\
 & \quad+  \beta_{i} \int_{\R_{0}^{n}} \Big[ (\lan x+ \gamma(x,i,z), G (x+ \gamma(x,i,z))\ran)^{\frac p2} - (x' Gx)^{\frac p2}  \\
   & \qquad \qquad \qquad \qquad \qquad  - p (x' Gx)^{\frac p2 - 1} \lan \gamma(x,i,z), Gx\ran\Big] \nu(\d z) \\
   & \leq  \sum_{j=1}^{m} q_{ij} \beta_{j}  (x' Gx)^{\frac p2}  + \beta_{i}  \bigl[ p\varrho_{i}  + p (p-2) \delta_{i} + \lambda_{i}  \bigr]  (x' Gx)^{\frac p2} \\
   & = \Biggl[ \sum_{j=1}^{m} q_{ij} \beta_{j}  - \theta_{i}\beta_{i} \Biggr] (x' Gx)^{\frac p2}   
    = -\bar \beta_{i}  (x' Gx)^{\frac p2}  \le - \varsigma V(x,i),
\end{align*}
where $\varsigma =  \min_{1 \le i \le m} \frac{\bar \beta_{i}}{\beta_{i}} > 0$. This verifies condition (ii) of  Theorem \ref{thm-moment-stab}.
Therefore by Theorem \ref{thm-moment-stab}, part (a), we conclude  that
 the equilibrium point of \eqref{sw-jump-diffusion} is exponentially stable in the $p$th moment. 
 
  The     assertion on a.s. exponential stability follows from   Theorem \ref{thm-moment-stab} part (b) for the case $p\in (0,2]$ and Theorem \ref{thm-moment-as-stable} for the case $p > 2$. \end{proof}

\section{Stability of Linear Markovian Regime-Switching Jump Diffusion  Systems}\label{sect-linear}
In this section, we consider a linear regime-switching jump diffusion
\begin{equation}
\label{linearsystem}
\begin{aligned}
\d X(t)  =  A(\al(t)) X(t)  \d t +    B(\al(t)) X(t)  \d W( t)
   + \int_{\R^{n}_{0}} C(\al(t-),z) X(t-)  \tilde N(\d t, \d z),
\end{aligned}\end{equation} where   
$\al(\cdot) $ is an irreducible    continuous-time Markov chain taking values in $\M = \{ 1, \dots, m\}$. Consequently we assume that $q_{ij}(\cdot)$ in \eqref{Q-gen} are constants for all $i,j \in \M$.
 In addition, unless otherwise mentioned, we assume that $\al(\cdot)$ has an invariant distribution $\pi = (\pi_{i}, i\in \M)$ throughout the section.
  In \eqref{linearsystem},
 for each $i \in \M$ and $z\in \R^{n}_{0}$, $A_{i} = A(i), B_{i}= B(i)$ and $C_{i}(z) = C(i,z)$ are $n\times n$ matrices satisfying the following condition
   \begin{equation}
\label{eq0-Ci(z)-cond}\begin{aligned}
 & \max_{i\in \M}\int_{\R_{0}^{n}}\bigl[ | C_{i}(z)|^{2}+ | C_{i}(z)|\bigr]\nu(\d z) < \infty,  \text{ and } \\
 & \lan \xi, (I  + C_{i}(z)')(I+C_{i}(z))\xi \ran \ge \varrho^{2}|\xi|^{2},   \text{ for all }\xi \in \R^{n} \text{ and }\nu\text{-almost all } z\in \R^{n}_{0},                        %\nu\{ z \in \R_{0}^{n}: I + C_{i}(z) \text{ is singular}\} =0, \text{ and }\\
\end{aligned}\end{equation}  where $\varrho$ is a positive constant. Apparently \eqref{linearsystem} satisfies Assumption \ref{assump-0}. In addition, the first equation of \eqref{eq0-Ci(z)-cond} guarantees that Assumption   \ref{assump-ito} is satisfied as well. Finally, since $|x+ C_{i}(z) x| = |(I+C_{i}(z)) x| = |\lan x, (I  + C_{i}(z)')(I+C_{i}(z))x \ran|^{\frac12} $,   the uniform ellipticity condition on the matrix $(I  + C_{i}(z)')(I+C_{i}(z))$ in the second equation of  \eqref{eq0-Ci(z)-cond} implies that Assumption \ref{assump-jump} holds.
% Notice that under the above conditions, Assumptions \ref{assump-0},   \ref{assump-ito}  and \ref{assump-jump} are satisfied.

We will deal with almost sure exponential stability in Section \ref{subsect:as-stab-linear} and moment exponential stability in Section \ref{subsect:pth-moment-stab-linear}. In both sections, we will treat   one-dimensional and multidimensional systems separately. Finally Section \ref{sect:examples} presents several examples. 

\subsection{Almost Sure Exponential Stability}\label{subsect:as-stab-linear}
\subsubsection{One Dimensional System} Let us first consider   the one-dimensional regime-switching jump diffusion
 \begin{equation}
\label{eq-SDE-1}
\begin{aligned}
\d x(t)  =  a(\alpha(t)) x(t)  \d t +    b(\alpha(t)) x(t)  \d W( t)
   + \int_{\R_{0}} c(\alpha(t-),z) x(t-)  \tilde N(\d t, \d z).
\end{aligned}\end{equation} %where $\al$ is an irreducible  continuous-time Markov chain with an invariant distribution $\pi = (\pi_{i}, i\in \M)$.
Suppose for each $i\in \M$, $a_{i} = a(i)$, $b_{i} =b(i)$ are real numbers,  and $c_{i}(\cdot) = c(i,\cdot)$  is a measurable function from $\R_{0}$ to $\R-\{-1\}$ satisfying \eqref{eq0-Ci(z)-cond}. Notice that \eqref{eq-SDE-1} can be regarded as an extended jump type Black-Scholes model with regime switching; this is motivated by the jump diffusion models   in   \cite{Cont-Tankov,cont2005} as well as the regime-switching models as in \cite{BW87,Zhang}. %
  %   the new twist is the inclusion of    regime switching mechanism to accommodate  the changes in the external economic environments. 
 \begin{prop}\label{prop-as-exp-stable-1d}    Suppose  \begin{equation}
\label{eq-1d-nu-condition}
 \max_{i\in \mathcal{M}} \biggl\{ \int_{\R_0}\big(\log|1+c_i(z)|)^2\nu(\d z) + \int_{\R_0}\bigl|\log|1+c_i(z)|-c_i(z)\bigr| \nu(\d z) \biggr\} <\infty,
\end{equation}then the solution to \eqref{eq-SDE-1} satisfies the following property:
      %\footnote{What will be $\lim_{t\rightarrow+\infty}\frac{1} {t}\log( \E[|x(t)|])$? See Theorem 5.24 on p. 184 of \cite{MaoY}.}
 \begin{equation*}\label{eq-Ci(z)-cond}
\aligned
\lim_{t\rightarrow+\infty}\frac{1} {t}\log|x(t)|=\delta: = \sum_{i=1}^m\pi_i\biggl[a_i-\frac12b^2_i+\int_{\R_0}\big(\log|1+c_i(z)|-c_i(z)\big)\nu(\d z)\biggr]\;\text{ a.s.}
\endaligned
\end{equation*}
%or
% \begin{equation}\label{eq-Ci(z)-cond}
%\aligned
%&\, \frac{1} {2} \lambda_{\max}\left(A_i+ A_i'+B_{i}' B_{i}\right)-\frac{1}{4}\min \big[\lambda(B_{i}+ B_{i}')\big]^{2}\\
%&+\int_{\R^{n}_{0}}\biggl\{\frac{\lambda_{\max}(C_i(z)'C_i(z))}{\lambda_{\min}[(I+\theta C_i(z))'(I+\theta C_i(z))]}\\
%&-\frac14\biggl[\frac{\mathcal{{\lambda}}[(I+\theta C_i(z)') C_i(z)+C_i(z)'(I+\theta C_i(z))]}{\lambda_{\max}[(I+\theta C_i(z))'(I+\theta C_i(z))]}\biggr]^2\biggr\}\nu(dz)<-\delta, \,\forall\,\theta\in(0,1),\text{ a.s.}
%\endaligned
%\end{equation}
%where $\mathcal{{\lambda}}[(I+\theta C_i(z)') C_i(z)+C_i(z)'(I+\theta C_i(z))]:=|\lambda_{\min}[(I+\theta C_i(z)') C_i(z)+C_i(z)'(I+\theta C_i(z))]
%|\vee|\lambda_{\max}[(I+\theta C_i(z)') C_i(z)+C_i(z)'(I+\theta C_i(z))]|$ and $\widetilde C_i:=\int_{\R^{n}_{0}}C_i(z)\nu(\d z).$
%\par
In particular,  the equilibrium point of \eqref{eq-SDE-1} is almost surely exponentially stable if and only if $\delta<0$.
\end{prop}

\begin{proof} As in the proof of Theorem \ref{thm-as-exp-stable}, we need only to consider the case when $x(t) \neq 0$ for all $t \ge 0$ with probability 1.
 %For $x\not = 0$, for each $i\in \M$, define $V(x,i)=\log|x|$.
 Let $x(0) = x\not =0$ and $\al(0) = i\in \M$.  Then by It\^o's formula we have
\begin{equation*}\label{wk1}
\aligned
 \log|x(t)| &  =  \log |x| +
   \int_0^t\biggl[a(\alpha(s))-\frac{1}{2}b^2(\alpha(s))\\ & \qquad\qquad \qquad \hfill +\int_{\R_0}[\log|1+c(\alpha(s-),z)|-c(\alpha(s-),z)]\nu(\d z)\biggr]\d s
        + M_{1}(t) + M_{2}(t),
   % \\ &+\int_0^t b(\alpha(t))dw(s)+\int_0^t\int_{\R_0}\log|1+c(\alpha(t),z)|\widetilde{N}(ds,dz)\\
 % =: &\,\log|x_0|+\int_0^t\biggl[a(\alpha(s))-\frac12b^2(\alpha(t))+\int_{R_0}[\log|1+c(\alpha(s-),z)|-c(\alpha(s-),z)]\nu(dz)\biggr]ds\\
% &+M_1(t)+M_2(t).
\endaligned
\end{equation*}
where \begin{align*}
M_1(t)  = \int_0^t b(\alpha(s))\d W(s),  \text{ and }
M_{2}(t)  = \int_0^t\int_{\R_0}\log|1+c(\alpha(s-),z)|\widetilde{N}(\d s,\d z).
\end{align*}

Obviously $M_1$ % and $M_2(t)$  are
is a  martingale vanishing at 0 with quadratic variation
$
\langle M_1, M_1\rangle_{t}=\int_0^tb^2(\alpha(s))\d s\leq t \max_{i\in \M} b_{i}^{2}.
$
On the other hand, % $\max_{i\in \mathcal{M}}\int_{\R_0}(\log|1+c_i(z)|)^2\nu(dz)<\infty$
\eqref{eq-1d-nu-condition}  implies that $M_{2}$ is a martingale vanishing at $0$. % see, for example, \cite{Applebaum}.
In addition, the quadratic variation of $M_{2}$ is given by
%\footnote{Does \eqref{eq0-Ci(z)-cond} implies that $\int_{\mathbb{R}_0}(\log|1
% + c_{i}(z)|)^{2}\nu(\d z) < \infty$ for each $i \in \M$?.
%\par
%Analysis: $|1+c_i(z)|>0$ yields
%$
%\log|1+c_i(z)|\leq|1+c_i(z)|-1\leq |c_i(z)|,
%$ but we can not obtain $
%(\log|1+c_i(z)|)^2\leq |c_i(z)|^2.
%$
%In fact, $\log |1+c_i(z)|\rightarrow-\infty$, while $c_i(z)\rightarrow -1.$
%Thus \eqref{eq0-Ci(z)-cond} doesn't imply  $\int_{\mathbb{R}_0}(\log|1
% + c_{i}(z)|)^{2}\nu(\d z) < \infty.$
%Also
%$\int_{\R_0}|c_i(z)|^2\nu(dz)<\infty$ does not imply $\int_{\R_0}|c_i(z)|\nu(dz)<\infty$.
%\par So I think here we need to assume that $\int_{\R_0}[\log|1+c_i(z)|-c_i(z)]\nu(dz)<\infty$ and  $\int_{\mathbb{R}_0}(\log|1
% + c_{i}(z)|)^{2}\nu(\d z) < \infty$.
%}
\[\aligned
\langle M_2, M_2\rangle_{t}=\int_0^t\int_{\mathbb{R}_0}(\log|1
 + c(\alpha(s-),z)|)^{2}\nu(\d z)\d s\le Kt,
\endaligned
\] where $K= \max_{i\in \M} \int_{\R_{0}} (\log |1+c_{i}(z)|)^{2}\nu (\d z) < \infty$.
Therefore we can apply  the strong law of large numbers for martingales (see, for example,  \cite[Theorem 1.6]{MaoY}) to conclude
\[
\lim_{t\rightarrow+\infty}\frac{M_1(t)}t=0 \,\,\; \text{and}\,\;\lim_{t\rightarrow+\infty}\frac{M_2(t)}t=0\; \text{ a.s.} \]
Then the ergodic theorem for continuous-time Markov chain  leads to the desired assertion.
\end{proof}

\subsubsection{Multidimensional Systems} Let us now focus on the multidimensional system  \eqref{linearsystem}.
As before, we
suppose that  the discrete component $\al$  in \eqref{linearsystem} and \eqref{eq:jump-mg} is an irreducible continuous-time Markov chain with  an invariant distribution $\pi = (\pi_{i}, i \in \M)$.
For notational simplicity, define the column
vector $\mu=(\mu_1,\mu_2,\dots,\mu_m)'\in\R^{m}$ with
\begin{equation}
\label{eq-mu-i-defn}
 \mu_i := \mu_{i}(G)
={1 \over 2}\lambda_{\max} (GA_i G^{-1} + G^{-1}A_i'G + G^{-1} B_{i}' G^{2} B_{i} G^{-1}),
\end{equation} where $G\in \S^{n\times n}$
is a positive definite matrix.  Also let
 \begin{equation}
\label{eq-beta-defn}
 \beta:=-\pi\mu= - \sum^m_{i=1} \pi_i \mu_i.
\end{equation}
%Note that $\beta>0$ by ?.
Then it follows from  Lemma A.12 of \cite{YZ-10}
 that the equation \begin{equation}
\label{eq-zeta-defn}
  Q \zeta= \mu +
\beta\one
\end{equation}  has a solution $\zeta=(\zeta_1,\zeta_2,\dots,\zeta_m)'\in\R^{m}$, where $\one: = (1,1,\dots, 1)' \in \R^{m}$.
Thus we   have from \eqref{eq-zeta-defn} that \begin{equation}\label{eq-com-0}\mu_i- \sum_{j=1}^m
  q_{ij}\zeta_j = -\beta,\quad i\in \M.\end{equation}
  
Before we state the main result of this section, let us introduce some more notation.  If $A$ is a square matrix, then $\rho(A)$ denotes the spectral radius of $A$.  Furthermore, if $A$ is symmetric, we denote
\begin{align}
\label{eq-lambda-hat-defn}
 &   \widehat{\lambda}(A):= %\left\{   \begin{array}{ll}
 \begin{cases}
   \lambda_{\max}(A), & \hbox{if } \lambda_{\max}(A)<0, \\
0\vee\lambda_{\min}(A),  &  \hbox{otherwise}.
  \end{cases}\end{align}

\begin{prop}\label{prop-exp-stab-creterion-b}
   %Assume that  the discrete component $\al$  in \eqref{linearsystem} and \eqref{eq:jump-mg} is an irreducible continuous-time Markov chain with  an invariant distribution $\pi = (\pi_{i}, i \in \M)$.
%Let $G\in \S^{n\times n}$  be a positive definite matrix and denote $\mu, \zeta\in \R^{m}$ and $\beta\in \R$ as above.
The  trivial solution of   \eqref{linearsystem} is a.s. exponentially stable if  there exist  a positive definite matrix $G\in \S^{n\times n}$, positive numbers $h_i $ and %$0<p\leq 2$
$p$ such that $h_i- p \zeta_{i} >0$ for each $i\in \M$ and that
\begin{equation}\label{eq-exp-stab-creterion}
\sum_{i\in \M} \pi_{i}  [c_{2}(i) -0.5c_3(i) - c_{4}(i)  -c_{5}(i)]< 0,
\end{equation} where %\footnote{We may need to prove/assume that the integral terms are well-defined. }
    % \footnote{ If $h_i=1$ then the last but one term is zero, which is corresponding with the condition obtained by $V(x,i) =
     % (1-p\zeta_i)(x'G^2 x)\sp {p /2}$, the second term has been added by estimating the term
     % $\frac{p-2}2\biggl(\frac{ x' B_{i}' G^{2} x}{x'G^2 x}\biggr)^{2}$ with $0<p\leq2$, which will reduce the value of $c_2$ to be negative.}
\begin{align}
\nonumber c_{2}(i)&: =  p\mu_i +\frac{p-2}8 {\Lambda}^2(GB_iG^{-1}+G^{-1}B_i'G)- \frac{p}{h_i-p\zeta_{i}} (\mu_{i}+\beta)+ \frac{1}{h_i-p\zeta_i}\sum_{j\in \M} q_{ij} h_{j}+ \eta_{i} ,
\\
\nonumber c_{3}(i)&: =\frac{ p^2}4\widehat{\lambda}^2(GB_iG^{-1}+G^{-1}B_i'G),\\
\label{eq-c2-5}  c_{4}(i)& : =  -  \biggl\{ 0\wedge\int_{\R_{0}^{n}}   \biggl[\frac{p}{2}\log{\lambda_{\max}(G^{-1}(I+ C_{i}(z))' G^{2} (I + C_{i}(z))G^{-1})}  \\
\nonumber  & \qquad \qquad\qquad \ - \bigl(  \lambda_{\min}(G^{-1}(I+ C_{i}(z))' G^{2} (I + C_{i}(z))G^{-1})\bigr)^{\frac{p}{2}}+1\biggr]\nu(\d z)\biggr\},  \\
\nonumber    c_{5}(i)& : = -  \sum_{j\in \M} q_{ij} \left(\log (h_j- p \zeta_{j})    - \frac{h_j-p \zeta_{j}}{h_i- p \zeta_{i}}\right),
   \end{align}
  in which $\mu_{i}$, $\beta$, and $\zeta_{i}$ are defined in \eqref{eq-mu-i-defn}, \eqref{eq-beta-defn}, and \eqref{eq-zeta-defn}, respectively, and
\begin{align}
  \label{eq-rho-i-defn} & \eta_{i}: = \int_{\R_{0}^{n}} \Bigl[   \bigl (  {\lambda_{\max}( G^{-1} (I + C_{i}(z) )' G^{2} (I + C_{i}(z)) G^{-1}}\bigr)^{\frac{p }{2}} - 1 \\
 \nonumber   & \qquad  \qquad\qquad
-   \frac{p}{2}  \lambda_{\min} (G C_{i}(z) G^{-1} + G^{-1} C_{i}'(z) G)\Bigr] \nu(\d z), \\
 &\label{eq-Lambda-defn}  \Lambda(GB_iG^{-1}+G^{-1}B_i'G):= \begin{cases}
     \widehat{\lambda}(GB_iG^{-1}+G^{-1}B_i'G),  & \text{ if } 0 < p \le 2, \\
     \rho (GB_iG^{-1}+G^{-1}B_i'G),  &  \text{ if } p >  2.
\end{cases}
\end{align}  In the above, we require that the integrals with respect to $\nu$ in \eqref{eq-c2-5} and \eqref{eq-rho-i-defn} are well-defined.
\end{prop}

\begin{rem}\label{rem-about-zeta-choice} Note that the constants $c_{2}(i)$ and $c_{5}(i)$ in the statement of Proposition \ref{prop-exp-stab-creterion-b} actually depend on the choice of the solution $\zeta$ to  equation \eqref{eq-zeta-defn}.  Nevertheless, for notational simplicity, we write $c_{2}(i), c_{5}(i)$ instead of $c_{2}(i;\zeta), c_{5}(i;\zeta)$.  Since $Q$ is a singular matrix,  and $\pi(\mu+ \beta \one) =0$, in view of  Lemma A.12 of \cite{YZ-10}, \eqref{eq-zeta-defn} has infinitely many solutions and any two solutions $\zeta^{1}, \zeta^{2}$ of \eqref{eq-zeta-defn} satisfy $\zeta^{1}-\zeta^{2} = \varrho \one$ for some $\varrho \in \R$.  Hence
 % the conclusion of
 Proposition \ref{prop-exp-stab-creterion-b} and in particular \eqref{eq-exp-stab-creterion} can be strengthened as: If
\begin{equation}
\label{eq1-exp-stab-creterion}
 \min\set{\sum_{i\in \mathcal{M}}\pi_i\big[c_2(i)-0.5 \,c_3(i)-c_4(i)-c_5(i)\big]\Big| \zeta\in \R^{m}, Q\zeta=\mu +
\beta\one}<0,
\end{equation} then   the trivial solution of   \eqref{linearsystem} is a.s. exponentially stable.
\end{rem}

The proof of Proposition \ref{prop-exp-stab-creterion-b} follows from a direct application of Theorem \ref{thm-as-exp-stable} and Corollary \ref{coro-as-exp-stable}. The idea is to construct an appropriate Lyapunov function $V$ that satisfies conditions (i)--(v) of Theorem \ref{thm-as-exp-stable}.  To preserve the flow of presentation, we arrange the proof to the Appendix \ref{sect-appendix}.

Next we present a sufficient condition for a.s. exponential stability for  the equilibrium point of a linear stochastic differential equation {\em without switching}. 
    \begin{coro}\label{coro1-as.exp-stab}
Let $i\in \M$. Suppose there exist  a positive definite matrix $G_{i}\in \S^{n\times n}$ and a positive number $p\in (0, 2]$ such that \begin{equation}
\label{eq-SDEi-exp-stable}
  \tilde{c}_{2}(i) -0.5c_3(i) - c_{4}(i)  < 0,
\end{equation} where $c_{3}(i), c_{4}(i)$ are defined in \eqref{eq-c2-5}, and
\begin{align*}
\tilde c_{2}(i)&: = p \mu_{i}+\frac{p-2}8{\Lambda}^2(G_{i}B_iG_{i}^{-1}+G_{i}^{-1}B_i'G_{i})   + \eta_{i},
%    {p \over 2}\lambda_{\max} (GA_i G^{-1} + G^{-1}A_i'G + G^{-1} B_{i}' G^{2} B_{i} G^{-1}) +\frac{p-2}8\widehat{\lambda}^2(GB_iG^{-1}+G^{-1}B_i'G)
% \\ & \ \qquad  + \int_{\R_{0}^{n}} \Bigl[   \bigl (  {\lambda_{\max}( G^{-1} (I + C_{i}(z) )' G^{2} (I + C_{i}(z)) G^{-1}}\bigr)^{\frac{p }{2}} - 1 \\ & \qquad  \qquad \qquad -   \frac{p}{2}  \lambda_{\min} (G C_{i}(z) G^{-1} + G^{-1} C_{i}'(z) G)\Bigr] \nu(\d z),
   \end{align*}
where $\mu_{i}, \eta_{i}$, and  ${\Lambda}(G_{i}B_iG_{i}^{-1}+G_{i}^{-1}B_i'G_{i})$ are similarly  defined in \eqref{eq-mu-i-defn},   \eqref{eq-rho-i-defn} and \eqref{eq-Lambda-defn} respectively,
 % \begin{align*}   &   \widehat{\lambda}(GB_iG^{-1}+G^{-1}B_i'G):=
 % \begin{cases}
 %  \lambda_{\max}(GB_iG^{-1}+G^{-1}B_i'G), & \hbox{if}\;\;\lambda_{\max}(A)<0, \\
% 0\vee\lambda_{\min}(GB_iG^{-1}+G^{-1}B_i'G), & \hbox{otherwise.}
%  \end{cases}
%\end{align*}
then the equilibrium point of the stochastic differential equation
\begin{equation}
\label{eq-SDEi}
\begin{aligned}
\d X^{(i)}(t)  =  A(i) X^{(i)}(t)  \d t +    B(i) X^{(i)}(t)  \d W( t)
   + \int_{\R^{n}_{0}} C(i,z) X^{(i)}(t-)  \tilde N(\d t, \d z),
\end{aligned}\end{equation}
is a.s. exponentially stable.

In addition, if $G_{i}= G$ and \eqref{eq-SDEi-exp-stable} holds for every $i\in \M$, then the equilibrium point of \eqref{linearsystem} is a.s. exponentially stable.
\end{coro}
\begin{proof}
This follows from Proposition \ref{prop-exp-stab-creterion-b} directly. % by considering the Lyapunov function $V_{i}(x) = (x' G_{i}^{2}x)^{\frac{p}{2}}$ and we shall omit the details here.
\end{proof}

% \begin{coro}\label{coro2-as.exp-stab}
% Let $i\in \M$. Suppose there exists a positive number $p\in (0, 1]$ such that
% \begin{equation} \label{eq2-SDEi-exp-stable}
%  \tilde{c}_{2}(i) -0.5c_3(i) - c_{4}(i)  < 0,
% \end{equation}
%then the equilibrium point of \eqref{eq-SDEi} is a.s. exponentially stable. In addition, if \eqref{eq2-SDEi-exp-stable} holds for every $i\in \M$, then the equilibrium point of \eqref{linearsystem} is a.s. exponentially stable.
% \end{coro}

% \begin{proof}
% For the first assertion, consider the Lyapunov function $V_{i}:= |x|^{p}$ and apply Corollary \ref{coro1-as.exp-stab}. The second assertion follows from Proposition \ref{prop-exp-stab-creterion-b}, in which we take the Lyapunov function $V(x,i) : = |x|^{p}$.
% \end{proof}

 %   Similarly, we have the following proposition for the one dimension case.

%\begin{coro}\label{coro-as-exp-p-stable}
% In addition to the conditions of Theorem \ref{prop2-as-exp-stable}, suppose also that the discrete component $\al$  in \eqref{linearsystem} and \eqref{eq:jump-mg} is an irreducible continuous-time Markov chain with  an invariant distribution $\pi = (\pi_{i}, i \in \M)$, then \eqref{wk1} can be relax to that there exists positive constant $\delta$ such that
% \begin{equation*}
% \sum_{i\in \mathcal {M}}\pi_i K_i<-\delta,\text{ a.s.}
%\end{equation*}
% \end{coro}

\subsection{$p$th Moment Exponential Stability}\label{subsect:pth-moment-stab-linear}
\subsubsection{One-Dimensional System}  As in Section \ref{subsect:as-stab-linear}, let us first derive a necessary and sufficient condition for the $p$th moment exponential stability for the one-dimensional linear system  \eqref{eq-SDE-1}. To this end, we need to introduce some notations.
Let $\mathcal P$ be the set of probability measures on the state space $\M$; then under the irreducibility and ergodicity assumptions,  the empirical measure of the  continuous-time Markov chain $\al(\cdot)$ satisfies the large deviation principle with  the rate function
\begin{equation}
\label{eq-rate-fn-MC}
\mathbf I(\mu) : = -\inf_{u_{1}, \dots, u_{m} > 0} \sum_{i,j\in \M} \frac{\mu_{i} q_{ij} u_{j}}{u_{i}},
\end{equation} where $\mu=(\mu_{1}, \dots, \mu_{m})\in \mathcal P$; we refer to \cite{DonsV-75} for details.  It is known that  $\mathbf I(\mu)$ is lower semicontinuous and  $\mathbf I(\mu)=0$ if and only if  $\mu= \pi$.
In addition, by virtue of \cite{Zong-14}, if $a=(a_{1}, \dots, a_{m})' \in \R^{m}$, then we have
\begin{equation}
\label{eq-Lambda(a)}
\Upsilon(a):= \lim_{t\to \infty} \frac{1}{t} \log\bigg(\E\bigg[\exp\bigg\{\int_{0}^{t} a(\al(s)) \d s \bigg\}\bigg]\bigg) = \sup_{\mu \in \mathcal P} \set{\sum_{i\in \M} a_{i} \mu_{i} -\mathbf I(\mu)}.
\end{equation}
Note that $\sum_{i \in \M}a_{i}\pi_{i} \le \Upsilon(a) \le \max_{i\in \M} a_{i}.$

\begin{prop}\label{prop-moment-exp-stab-1d}
Assume the conditions of Proposition \ref{prop-as-exp-stable-1d}. In addition, assume that there exists some $p>0$ such that for each $i\in \M$, $\int_{\R_{0}}  \abs{|1+ c(i,z)|^{p} - p c(i,z) -1}\nu(\d z) < \infty.$ Denote  $f=(f(1), \dots, f(m))$ with
\begin{displaymath}
 f(i)  = f_{p}(i) = pa(i)+ \frac{1}{2}p(p-1)b^2(i)    + \int_{\R_{0}}  [|1+ c(i,z)|^{p} - p c(i,z) -1]\nu(\d z).
\end{displaymath}  Then we have  \begin{equation}
\label{eq-Lypunov exponent}
\lim_{t\to \infty} \frac{1}{t} \log(\E[|x(t)|^{p}]) = \Upsilon(f),
\end{equation}  where $\Upsilon(f)$ is similarly defined as in \eqref{eq-Lambda(a)}. Therefore, the trivial solution of  \eqref{eq-SDE-1} is $p$th moment exponentially stable if and only if $\Upsilon(f) < 0$.
\end{prop} \begin{proof}
See Appendix \ref{sect-appendix}.
\end{proof}

\subsubsection{Multidimensional System} Now let's focus on establishing a sufficient condition for the $p$-th moment exponential stability of the trivial solution of the multidimensional system \eqref{linearsystem}.
In view of Theorem \ref{thm-moment-stab} and the calculations in Proposition \ref{prop-exp-stab-creterion-b}, we have the following proposition:
\begin{prop}\label{cor-p-stab}
If there exist a positive definite matrix $G\in \S^{n\times n}$, positive numbers $p$ and $h_{i}, i \in \M$ such that
\begin{equation}
\label{eq-pth-moment-stab-condition}
\delta: = \min\biggl\{ \max_{i\in \M} c(i; h, \zeta)\Big| \zeta \in \R^{m}, Q\zeta = \mu + \beta \one, h_{i} - p \zeta_{i} > 0 \text{ for each } i \in \M \biggr \} < 0,
\end{equation}  then  the equilibrium point of \eqref{linearsystem} is  exponentially stable in the $p$th moment with Lyapunov exponent less than or equal to $\delta$, 
where $\mu \in \R^{m}$ and $\beta \in \R$ are defined  in \eqref{eq-mu-i-defn} and \eqref{eq-beta-defn}, respectively, and
\begin{displaymath}
 c(i; h, \zeta) : =   p\mu_i +\frac{p-2}8{\Lambda}^2(GB_iG^{-1}+G^{-1}B_i'G)- \frac{p}{h_i-p\zeta_{i}} (\mu_{i}+\beta)+ \frac{1}{h_i-p\zeta_i}\sum\limits_{j\in \M} q_{ij} h_{j}+ \eta_{i},
  \end{displaymath}
in which $\eta_{i}$ and  ${\Lambda}(GB_iG^{-1}+G^{-1}B_i'G)$    are defined in   \eqref{eq-rho-i-defn} and \eqref{eq-Lambda-defn}, respectively.
   %   and $ \rho(GB_iG^{-1}+G^{-1}B_i'G)$ denotes the spectral radius of the matrix $GB_iG^{-1}+G^{-1}B_i'G$.

\end{prop}

\begin{proof} Let $p, h= (h_{1}, \dots, h_{m})' $ and $G$ be as in the statement of the corollary and
 consider the function $V(x,i)  = (h_{i} - p \zeta_{i}) (x' G^{2} x)^{p/2} $, $(x,i) \in \R^{n }\times \M$. Then we have
 \begin{displaymath}
0 < \min_{i\in \M} (h_{i} - p \zeta_{i})( \lambda_{\min}(G^{2}))^{p/2}  |x|^{p}\le V(x,i)  \le \max_{i\in \M} (h_{i} - p \zeta_{i})( \lambda_{\max}(G^{2}))^{p/2} |x|^{p}.
\end{displaymath}
Moreover, the detailed calculations in the proof of Proposition \ref{prop-exp-stab-creterion-b} reveal that  \begin{displaymath}
\op V(x,i) \le c(i; h, \zeta) V(x,i)  \le \max_{i \in \M}c(i; h, \zeta)  V(x,i).
\end{displaymath} Then condition \eqref{eq-pth-moment-stab-condition} and Theorem  \ref{thm-moment-stab} lead to the conclusion.
\end{proof}

Finally we apply Theorem \ref{thm-moment-stab-m-matrix} to derive a sufficient condition for a.s. and moment exponential stability for the  equilibrium point of \eqref{linearsystem}. Note that in Proposition \ref{prop-as-moment-stab-M-matrix} below, the infinitesimal generator $Q=(q_{ij})$ of the continuous-time Markov chain $\al$ need not to be irreducible and ergodic.  
\begin{prop}\label{prop-as-moment-stab-M-matrix}
Suppose that there exist a positive constant $p$ and a positive definite matrix $G\in \S^{n\times n}$ such that the $m\times m$ matrix $\mathcal A:=\mathrm{diag}(\theta_{1},\dots \theta_{m}) -Q$ is a nonsingular $M$-matrix, then the equilibrium point of \eqref{linearsystem} is    $p$th moment exponentially stable, where for each $i\in \M$, 
\begin{align*}
  \theta_{i}  &: = -\biggl[\frac{p}{2}\lambda_{\max}(GA_{i}G^{-1} +G^{-1}A_{i}G + G^{-1}B_{i}' G^{2} B_{i}G^{-1}) \\
    & \qquad\qquad + \frac{p(p-2)}{4} \rho^{2}(GB_{i}G^{-1} +G^{-1}B_{i}G)+ \eta_{i} \biggr ],     
\end{align*} and $ \eta_{i}$ is defined in \eqref{eq-rho-i-defn}. If in addition that either $p \in (0, 2]$ or else $p > 2$ with $$\max_{i\in\M}\int_{\R_{0}^{n}}|C_{i}(z)|^{p}\nu(\d z) < \infty,$$  then the equilibrium point of \eqref{linearsystem} is also a.s. exponentially stable. 
\end{prop} The proof of Proposition \ref{prop-as-moment-stab-M-matrix} consists of  straightforward verifications of \eqref{eq1-moment-stable-M}--\eqref{eq3-moment-stable-M} of Assumption \ref{assump-m-matrix}. Theorem  \ref{thm-moment-stab-m-matrix}  then leads to the assertions on almost sure and moment stability. Again we shall arrange the proof to the appendix \ref{sect-appendix}.

\subsection{Examples}\label{sect:examples}

\begin{example}
In this example, we consider  the one-dimension linear system given in  \eqref{eq-SDE-1}, in which  $\alpha \in \M = \{ 1,2,3\}$ is a continuous-time Markov chain with generator
$Q= \begin{bmatrix}
        -3 & 1 & 2\\
       2 & -2 & 0\\
       4 & 0 & -4 \\
      \end{bmatrix},$    $a_{1}=4,a_{2} = 2,a_{3}=3$, $b_{1}=1,b_{2}=3, b_{3}=1 $, and  $c_{i}(z)=1\wedge z^2 $ for $i=1,2,3$.
 In addition, suppose that the characteristic measure of the Poisson random measure $N$ is given by the  L\'evy measure  $\nu (\d z) =  \frac{\d z}{ z^{4/3}}$, $z\in \R_{0}$. Note that      $\nu$ is an infinite L\'evy measure, i.e., $\nu(\R_{0}) =\infty$. 
     %$0 < r< 1$. For simplicity, we choose
% in which $r=\frac13$.  %Consider. In \eqref{eq-SDE-1}, let

By direct computations, we get
 % \[\pi=(0.5, 0.25, 0.25),\int_{\R_0}\log|1+c_i(z)|\nu(\d z)=5.1633,\int_{\R_0}c_i(z)\nu(\d z)= 7.2.  \]
% Thus it follows that
% \[\int_{\R_0}\big[\log|1+c_i(z)|-c_i(z)\big]\nu(\d z)= -2.0367, \text{ for }i =1,2,3,\]
%\[ a-\frac12b^2+\int_{\R_0}\big(\log|1+c(z)|-c(z)\big)\nu(\d z)=(1.4633,   -4.5367,    0.4633)', \]
% and
\[
 \delta=\sum_{i=1}^3 \pi_i\biggl[a_i-\frac12b^2_i+\int_{\R_0}\big(\log|1+c_i(z)|-c_i(z)\big)\nu(\d z)\biggr]=  -0.2867.
\]
Then    Proposition  %\ref{prop2-as-exp-stable-2-2}  \ref{prop-exp-stab-creterion-b}
\ref{prop-as-exp-stable-1d} implies that  the trivial solution of  \eqref{eq-SDE-1} is almost surely exponentially stable.

However, if the  jumps are excluded from the system \eqref{eq-SDE-1}, that is, if  $c_i(z)=0$ for $i =1,2,3$, then
% \[
% a-\frac12b^2=(3.5,   -2.5,    2.5)'
%\]
% and
\[\sum_{i=1}^3 \pi_i \biggl(a_i-\frac12b_i^2\biggr)=1.75,\]
which implies that the trivial solution of \eqref{eq-SDE-1} %the one dimension linear equation without jumps
is   almost surely exponentially  unstable.
 This example indicates that the jumps can contribute to the stability of the equilibrium point.
\end{example}

\begin{example}\label{example1} Consider the linear system
   \begin{equation}
\label{eq-ex-linearsystem}
\begin{aligned}
\d X(t)  =  A(\al(t)) X(t)  \d t +    B(\al(t)) X(t)  \d W( t)
   + \int_{\R_{0}} C(\al(t-),z) X(t-)  \tilde N(\d t, \d z),
\end{aligned}\end{equation} in which $\alpha \in \M = \{ 1,2,3\}$ is a continuous-time Markov chain with generator $Q=      \begin{bmatrix}
        -3 & 1 & 2\\
       2 & -2 & 0\\
       4 & 0 & -4 \\
      \end{bmatrix},$ $N$ is a Poisson random measure on $[0,\infty)\times \R_{0}$ whose corresponding L\'evy measure is given by
      $\nu(\d z)=\frac12(e^{-z}\wedge e^z)\d z, z \in \R$,    and
 \begin{align*}
& A_1=\left[
      \begin{array}{ccc}
        10 & 1 & 8 \\
        -3 & 10 & 2 \\
        -1 & -8 & 12 \\
      \end{array}
    \right],\;\;
& A_2=\left[
      \begin{array}{ccc}
        17 &  5 &  8\\
       -1 &  11 &  -3\\
        4 &  -5 &  13\\
      \end{array}
    \right],\ \quad
& A_3=\left[
      \begin{array}{ccc}
       10 &  -4 &  12 \\
       8 &  10 &  -8\\
       3 &  -9 &  11\\
      \end{array}
    \right],\;\;\\
& B_1=  \left[
     \begin{array}{ccc}
         1 & 2 & 0  \\
        -2 & 1 & 4 \\
        -1 & -2 & 1 \\
     \end{array}
   \right],
& B_2=\left[
     \begin{array}{ccc}
       -1 &  2 & 1 \\
       -3 & 1 & 1 \\
       2 & -1 & 1 \\
     \end{array}
   \right],\ \quad
& B_3=\left[
     \begin{array}{ccc}
        1  & 2  & 2\\
       -1  & 1  & 4 \\
       -3  & -2  & 1 \\
     \end{array}
   \right],\\  &  C_{i}(z) = 0 \in \R^{3\times 3}. \end{align*}   
   Here $\nu$ belongs to the class of double exponential distributions; we refer to \cite{Kou-02} for applications of such distributions in math finance.   

 Let us take      \[ \aligned
G=& \left[
      \begin{array}{ccc}
      3 & 0 & 0\\
       0 & 2 & 0\\
       0 & 0 & 3 \\
      \end{array}
    \right],\;\;
% Q=\left[
 %     \begin{array}{ccc}
%        -3 & 1 & 2\\
 %      2 & -2 & 0\\
 %      4 & 0 & -4 \\
 %     \end{array}
 %   \right],\\
% C_i=&\mathbf{0},\;\;
h=[20, 20, 20],\;\;
p=0.1.
\endaligned
\]
Then by direct calculation, we get
\[
\aligned
\pi=&(0.5, 0.25, 0.25),\;\;\mu=(23.7194,   34.0899,   28.3542)',\ \
& \beta=  -27.4707,\;\;c_3 = c_4 = \mathbf{0},
\endaligned
\]
and \begin{equation}
\label{eq-ex1-calcuation}
  \min_{\zeta\in D}\sum_{i\in \M}\pi_i[c_2(i)-c_5(i)]=2.7422>0,
\end{equation} where $D:=\{\zeta=(\zeta_1, \zeta_2, \zeta_3)\in \R^3| \;Q\zeta=\mu+\beta\one\}$,
   and the %minimum point
   minimizer in \eqref{eq-ex1-calcuation} is given by $\zeta=(131.65,  112,  142.5264)'$, and
 \[\aligned
 &c_2(1)-c_5(1)=6.5169 ,\;\;c_2(2)-c_5(2)=9.8252,\;\;c_2(3)-c_5(3)= 2.7433.
 \endaligned
 \]
Thus we cannot apply Proposition \ref{prop-exp-stab-creterion-b} to determine the almost surely exponential stability of the trivial solution of \eqref{eq-ex-linearsystem}.

Next we observe that
\[\sum_{i\in \mathcal{M}}\pi_i\left[\lambda_{\min}(B_iB_i')+\frac12\lambda_{\min}(A_i+A_i')-\max\big\{\lambda_{\min}^2(B_iB_i'),\lambda_{\max}^2(B_iB_i')\big\}\right]=0.3541.\]
Therefore by virtue of   Theorem 4.3 in \cite{KZY07},   the trivial solution of \eqref{eq-ex-linearsystem} is unstable in probability, which, in turn, implies  that the trivial solution cannot be almost surely exponential stable.
\end{example}
\begin{example}\label{example2}
Again consider the linear system \eqref{eq-ex-linearsystem}  with the same $Q, A_{i}, B_{i}, G, h, $ and $p$ as given in Example \ref{example1}, but with
 % Now we include jumps by taking
\[C_1=\left[
      \begin{array}{ccc}
        15 & 1 & 2\\
       -1 & 9  & 1\\
       7  & -1 & 10\\
      \end{array}
    \right], \;\;\\C_2=\left[
      \begin{array}{ccc}
        20  & 6  & -3\\
       1  & 14 &  2\\
       3  & 2  & 8\\
      \end{array}
    \right], \;\;\\C_3=\left[
      \begin{array}{ccc}
        7 &  1  & 4\\
       2  & 10  & 1\\
       -1  & 5  & 11\\
      \end{array}
    \right].
\]
Then we have
\[
\aligned
\pi=&(0.5, 0.25, 0.25),\;\;\mu=(23.7194,   34.0899,   28.3542)',\ \ &
\beta= -27.4707,\;\;c_3 = c_4 = \mathbf{0},
\endaligned
\] and
\begin{equation}
\label{eq-ex2-calculation}
\min_{\zeta\in D}\sum_{i\in \mathcal{M}}\pi_i[c_2(i)-c_5(i)]=-0.0939,
\end{equation}
where the %minimum point
minimizer of \eqref{eq-ex2-calculation} is $\zeta=(116.2209,  112.9113,  116)'$, and
\[
\aligned
&c_2(1)-c_5(1)=-0.3687 ,\;\;c_2(2)-c_5(2)=-0.1647 ,\;\;c_2(3)-c_5(3)= 0.3463.
\endaligned
\]
Therefore thanks to  Proposition \ref{prop-exp-stab-creterion-b}, the trivial solution of %the above system
 \eqref{eq-ex-linearsystem}  is almost surely exponentially stable. A comparison between Examples \ref{example1} and \ref{example2} shows that in some cases, the jumps can suppress the growth of the solution. %Meanwhile,
 In addition, we notice that   the  switching mechanism  also %has contribution
 contributes to the almost surely exponential stability. % of the trivial solution of the above system.
\end{example}

\section{Conclusions and Further Remarks}\label{sect-conclusion}

Motivated by the emerging applications of complex   stochastic  systems in areas such as finance and energy spot price modeling, this paper  is devoted to almost sure and $p$th moment exponential stabilities of regime-switching jump diffusions. The main results include  sufficient conditions for almost sure and $p$th moment exponential stabilities of the equilibrium point of  nonlinear and linear regime-switching jump diffusions. For general nonlinear systems,  the sufficient conditions for stability are expressed in terms of the existence of appropriate Lyapunov functions; from which we also derive  a  condition using $M$-matrices. In addition,  we  show  that $p$th moment stability implies almost sure exponential stability.
 For one-dimensional linear  regime-switching jump diffusions, we obtain necessary and sufficient conditions for almost sure and $p$th moment exponential stabilities. For the multidimensional system, we present  verifiable sufficient conditions in terms of the eigenvalues of certain matrices for stability. Several examples are provided to illustrate the results.

In this work, the switching component $\alpha$ has a finite state space. A relevant question is: Can we allow $\alpha $ to have an infinite countable space? In addition,  the jump part is driven  by  a Poisson random measure associated with a L\'evy process.
 A worthwhile future effort is to treat systems in which
the random driving force is an alpha-stable process that has finite
$p$th moment with $p<2$. This requires more work and careful consideration.

\section*{Acknowledgements}  We would like to thanks the anonymous reviewers for their useful comments and suggestions. 

The research of Zhen Chao was supported in part by the NSFC (No. 11471122), the NSF of Zhejiang Province
(No. LY15A010016) and ECNU reward for Excellent Doctoral Students in Academics (No. xrzz2014020). The research of Kai Wang and Yanling Zhu was supported in part by the NSF
of Anhui Province (No. 1708085MA17 and No. 1508085QA13), the Key NSF of Education Bureau of Anhui Province
(No. KJ2013A003) and the Support Plan of Excellent Youth Talents in Colleges and Universities in Anhui Province (2014).
The research of Chao Zhu was supported in part by the NSFC (No. 11671034) and the Simons Foundation   (award number 523736).
\appendix
\section{Several Technical Proofs}\label{sect-appendix}

\begin{proof}[Proof of Theorem \ref{thm-as-exp-stable}]  %Suppose first  that  \eqref{eq2-jump-cond} holds. Then 
Recall that thanks to Lemma \ref{lem-0-inaccessible},  
  for every $(x_{0},\al_{0}) \in \R_{0}^{n}\times \M$,
  $X(t) := X^{x_{0},\al_{0}}(t) \not=0$ for all $t\geq 0$ a.s. Let $U(x,i)=\log V(x,i)$ for $(x,i)\in  \R^{n}_{0} \times \M$. Since
\begin{displaymath}
DU(x,i)=\frac{DV(x,i)}{V(x,i)} \text{ and } D^2U(x,i)= \frac{D^2V(x,i)}{V(x,i)}-\frac{DV(x,i) DV(x,i)'}{V^2(x,i)},
  \end{displaymath}we have
\begin{align}
\nonumber  \op & U(x,i)   \\ &  \nonumber = \frac{\langle D V(x,i), b(x,i)\rangle}{V(x,i)} + \frac{1}{2 V(x,i) } \tr \bigl ( \sigma\sigma'(x,i) D^{2} V(x,i) \bigr) - \frac{|\langle DV(x,i),\sigma(x,i)\rangle|^2}{2V^2(x,i)} \\
 \nonumber   &   \   + \sum_{j\in \M} q_{ij}(x)  \log V(x,j) + \int_{\R_0^n}\biggl[\log\frac{V(x +\gamma(x,i,z),i)}{V(x, i)}-\frac{  DV(x,i) \cdot \gamma(x,i,z)}{V(x,i)}\biggr]\nu(\d z)\d s \\
\nonumber     & = \frac{ \op V(x,i)}{V(x,i)}  - \frac{|\langle DV(x,i),\sigma(x,i)\rangle|^2}{2V^2(x,i)} + \sum_{j\in \mathcal{M}}q_{ij}(x)\left(\log V(x,j)-\frac{V(x,j)}{V(x,i)}\right)  \\  \label{eq-LU-calculation}
     & \   + \int_{\R_0^n}\biggl[\log\frac{V(x +\gamma(x,i,z),i)}{V(x, i)} - \frac{V(x +\gamma(x,i,z),i)}{V(x,i)}+1\biggr] \nu(\d z).
\end{align}

Now we
   apply It\^o's formula to the process $U(X(t), \al(t))$:
\begin{align}\label{eq-Ito-U}
  U& (X(t), \al(t)) =  U(x_{0},\al_{0})   +   \int_{0}^{t} \op U(X(s ),\al(s) ) \d s   + M(t), \end{align}
  where  $M(t) = M_{1}(t) + M_{2}(t) + M_{3}(t)$, and
  \begin{align*} M_{1}(t) & = \int_{0}^t\frac{\langle DV(X(s), \al(s)), \sigma(X(s),\al(s)) \rangle}{V(X(s),\al(s))} \d W(s), \\
   M_{2}(t) & = \int_0^{t}\int_{\R_0^n}\log \frac{V(X(s-)+\gamma(X(s-),\al(s-),z),\al(s-))}{V(X(s-),\al(s-))}\widetilde{N}(\d s,\d z), \\
   M_{3}(t) & = \int_0^{t}\int_{\R}\log \frac{V(X(s-),\al(s-) + h(X(s-), \al(s-), y ) )}{V(X(s-),\al(s-))}\widetilde{N}_{1}(\d s,\d y).
\end{align*}

By the exponential martingale inequality \cite[Theorem 5.2.9]{APPLEBAUM}, for any $k\in \mathbb N$ and $ \theta  \in ( 0, 1)$,  we have
\begin{align*}
 \P  \biggl\{  \sup_{ 0\le t \le k} \biggl[ & M (t) -  \frac{\theta} {2} \int_{0}^{t}\frac{|\langle DV(X(s), \al(s)), \sigma(X(s),\al(s))\rangle|^{2}}{|V(X(s),\al(s))|^{2}}   \d s  \\ & \qquad\qquad
 - f_{1,\theta}(t)- f_{2,\theta}(t) 
        \biggr]
      >  \theta \sqrt k \biggr\}
       \le   e^{-\theta^{2} \sqrt{k}},
\end{align*}  where \begin{align*}
    f_{1,\theta}(t) & =      \frac{1}{\theta} \int_{0}^{t}\int_{\R_{0}^{n} } \biggl[  \biggl(\frac{V(X(s-)+\gamma(X(s-),\al(s-),z),\al(s-))}{V(X(s-),\al(s-))}  \biggr)^{\theta} - 1 \\
       &\qquad  \qquad  - \theta \log  \frac{V(X(s-)+\gamma(X(s-),\al(s-),z),\al(s-))}{V(X(s-),\al(s-))} \biggr]\nu(\d z)\d s, \\
     f_{2,\theta}(t) &=    \frac{1}{\theta}  \int_{0}^{t}\int_{\R } \biggl[  \biggl(\frac{V(X(s-),\al(s-) + h(X(s-), \al(s-), y ) )}{V(X(s-),\al(s-))}  \biggr)^{\theta} - 1 \\
       &\qquad\qquad  -  \theta \log  \frac{V(X(s-),\al(s-) + h(X(s-), \al(s-), y ) )}{V(X(s-),\al(s-))}\biggr] \lambda(\d y) \d s.
      \end{align*}
We can verify that   $\sum_{k}  e^{-\theta^{2} \sqrt{k}}< \infty$. Therefore the Borel-Cantelli lemma implies that there exists an $\Omega_{0} \subset \Omega$ with $\P(\Omega_{0}) =1$ such that for every $\omega \in \Omega_{0}$, there exists an integer $k_{0} = k_{0}(\omega)$ so that for all $k \ge k_{0}$ and   $0 \le t \le k$, we have
\begin{align}\label{eq-M(t)-bound}
  M(t)    \le  &\,  \frac{\theta} {2} \int_{0}^{t}\frac{|\langle DV(X(s), \al(s)), \sigma(X(s),\al(s))\rangle|^{2}}{|V(X(s),\al(s))|^{2}}   \d s   +  \theta \sqrt k + f_{1,\theta}(t) +   f_{2,\theta}(t). \end{align}

   % Putting   conditions (ii)--(v) into \eqref{eq-LU-calculation} yields
%\begin{displaymath}
% \op U(x,i) \le c_{2}(i)  - \frac{c_{3}(i) }{2} - c_{4}(i)  -c_{5}(i) .
%\end{displaymath}

 Now putting \eqref{eq-M(t)-bound} and  \eqref{eq-LU-calculation} into \eqref{eq-Ito-U}, it follows that for all $\omega \in \Omega_{0}$ and $0 \le t \le k$, we have
 \begin{align}
\nonumber   U& (X(t), \al(t)) -  U(x_{0},\al_{0})\\
 \nonumber   & \le      \int_{0}^{t} \frac{\op V(X(s), \al(s))}{ V(X(s), \al(s))} \d s - \frac{1 -\theta}{2}     \int_{0}^{t}\frac{|\langle DV(X(s), \al(s)), \sigma(X(s),\al(s))\rangle|^{2}}{|V(X(s),\al(s))|^{2}}   \d s \\
\nonumber      & \quad +   \theta \sqrt k  + f_{1,\theta}(t) +   f_{2,\theta}(t)  +  \int_{0}^{t}   \sum_{j\in \mathcal{M}}q_{\al(s), j}(X(s))\biggl(\log V(X(s),j)-\frac{V(X(s),j)}{V(X(s),\al(s))}\biggr) \d s  \\
\nonumber    &  \quad +  \int_{0}^{t} \int_{\R_0^n}\biggl[\log\frac{V(X(s-) +\gamma(X(s-),\al(s-),z),\al(s-))}{V(X(s-), \al(s-))} + 1  \\
 \nonumber      & \qquad  \qquad \qquad - \frac{V(X(s-) +\gamma(X(s-),\al(s-),z),\al(s-))}{V(X(s-),\al(s-))}\biggr] \nu(\d z) \d s\\
  \label{eq-U(Xalpha(t))}     & \le  \int_{0}^{t} \biggl[c_{2}(\al(s))-  \frac{1 -\theta}{2}c_{3}(\al(s)) -c_{4}(\al(s)) -c_{5}(\al(s))\biggr]\d s+   \theta \sqrt k + f_{1,\theta}(t) +   f_{2,\theta}(t).
\end{align}

Next we argue that for any $t \ge 0$, $ f_{1,\theta}(t) +   f_{2,\theta}(t) \to 0$  as $\theta  \downarrow 0$.
To this end, we first use
  the elementary inequality $e^a\geq a+1$ for $a\in\R$ to obtain \begin{align*}
 \frac{1}{\theta}&  \biggl[  \biggl(\frac{V(x+\gamma(x,i,z),i )}{V(x,i )}  \biggr)^{\theta} - 1
 - \theta \log  \frac{V(x+\gamma(x,i,z),i )}{V(x,i )} \biggr] \ge 0\, \text{ for any }(x,i) \in \R^{n}\times \M.
\end{align*} Next the inequality  $x^r\leq 1+r(x-1)$ for $0\leq r\leq 1$ and  $x >0$ leads to
\begin{align*}
 \frac{1}{\theta}&  \biggl[  \biggl(\frac{V(x+\gamma(x,i,z),i )}{V(x,i )}  \biggr)^{\theta} - 1
 - \theta \log  \frac{V(x+\gamma(x,i,z),i )}{V(x,i )} \biggr] \\
 & \le  \frac{1}{\theta}  \biggl[ 1+ \theta \biggl(\frac{V(x+\gamma(x,i,z),i )}{V(x,i )} -1 \biggr) -1- \theta \log  \frac{V(x+\gamma(x,i,z),i )}{V(x,i )}  \biggr]  \\
 & = \frac{V(x+\gamma(x,i,z),i )}{V(x,i )} -1 - \log  \frac{V(x+\gamma(x,i,z),i )}{V(x,i )};
\end{align*}  notice that the last expression in the above equation is nonnegative thanks to the inequality $a - 1 - \log a \ge  0$ for $a> 0$. Next by virtue of \eqref{eq2-ito-condition},
    %\eqref{eq0-Ci(z)-cond},
    we can slightly modify the proof of  Lemma 3.3 in \cite{ApplS-09} to obtain
\begin{align*}
\int_{0}^{t} \int_{\R_{0}^{n} }   \biggl[& \frac{V(X(s-)+\gamma(X(s-),\al(s-),z),\al(s-))}{V(X(s-),\al(s-))}  - 1 \\
  &  - \log   \frac{V(X(s-)+\gamma(X(s-),\al(s-),z),\al(s-))}{V(X(s-),\al(s-))}  \biggr] \nu(\d z) \d s < \infty \; \text{ a.s.}
\end{align*}
 In addition, we can verify that $\lim_{\theta \downarrow 0} [\frac{1}{\theta} (a^{\theta} -1) - \log a] = 0$ for $a > 0$.
 Then   the dominated convergence theorem leads to
  \begin{align*}
\lim_{\theta \downarrow 0} f_{1,\theta}(t)  & = \int_{0}^{t} \int_{\R_{0}^{n}} \lim_{\theta \downarrow 0}  \frac{1}{\theta}   \biggl[  \biggl(\frac{V(X(s-)+\gamma(X(s-),\al(s-),z),\al(s-))}{V(X(s-),\al(s-))}  \biggr)^{\theta} - 1 \\
       &\qquad   \quad - \theta \log  \frac{V(X(s-)+\gamma(X(s-),\al(s-),z),\al(s-))}{V(X(s-),\al(s-))} \biggr]\nu(\d z)\d s = 0.
\end{align*}

On the other hand, using \eqref{eq-q-bdd}, we can readily verify that  \begin{align*}
  \int_{0}^{t}\int_{\R } \biggl[ &  \frac{V(X(s-),\al(s-) + h(X(s-), \al(s-), y ) )}{V(X(s-),\al(s-))}  - 1 \\
       &\qquad  -    \log  \frac{V(X(s-),\al(s-) + h(X(s-), \al(s-), y ) )}{V(X(s-),\al(s-))}\biggr] \lambda(\d y) \d s < \infty\; \text{ a.s. } \end{align*}
Therefore using exactly the same argument as above, we derive $\lim_{\theta \downarrow 0} f_{2,\theta}(t) =0 $.

Now passing to the limit as $\theta \downarrow 0$ in \eqref{eq-U(Xalpha(t))} leads to
\begin{equation}\label{eq-U-prelimit}
U  (X(t), \al(t)) -  U(x_{0},\al_{0})\le  \int_{0}^{t} \bigl[c_{2}(\al(s))-  0.5 c_{3}(\al(s)) -c_{4}(\al(s)) -c_{5}(\al(s))\bigr]\d s
\end{equation} for all $\omega \in \Omega_{0}$, $k\ge k_{0} = k_{0}(\omega) $ and $0 \le t \le k$.
Recall that $U(x,i) = \log V(x,i)$. Then inserting condition (i)  into \eqref{eq-U-prelimit} yields that  for almost all $\omega\in \Omega,$ $k\geq k_0 ,$ and $ k-1\leq t\leq k$, we have
 \begin{align}
\label{eq2-U-prelimit}
\nonumber   \frac{1}{t}&  [p \log|X(t)| + \log c_{1}(\al(t)) ] \le  \frac1t\log V(X(t),\al(t))\\
 \nonumber   & \le \frac{1}{t}  \int_{0}^{t} \bigl[c_{2}(\al(s))-  0.5 c_{3}(\al(s)) -c_{4}(\al(s)) -c_{5}(\al(s))\bigr]\d s  +\frac{\log V(x_{0},\al_{0})}{t}\\
 \nonumber   & \le \frac{1}{t}  \int_{0}^{t} \bigl[c_{2}(\al(s))-  0.5 c_{3}(\al(s)) -c_{4}(\al(s)) -c_{5}(\al(s))\bigr]\d s  +\frac{\log V(x_{0},\al_{0})}{k-1}\\
  & \le \max_{i\in \M}\bigl\{c_{2}(i)-  0.5 c_{3}(i) -c_{4}(i) -c_{5}(i) \bigr \} +\frac{\log V(x_{0},\al_{0})}{k-1};
\end{align} the last inequality yields \eqref{eq1-X-as-exp-stab} by letting $t\to \infty$. %  (and hence $t\to\infty$).
  \end{proof}

\begin{proof}[Proof of Theorem \ref{thm-moment-as-stable}]
%We use similar idea as that in the proof of  % Theorem 4.4 of \cite{ApplS-09} and  \cite[Theorem 5.9]{MaoY}.   
Fix some $(x_{0},\al_{0}) \in \R^{n}_{0}\times \M$ and denote by $(X(t),\al(t))$ the unique solution to \eqref{sw-jump-diffusion}--\eqref{eq:jump-mg}   with initial condition $(X(0),\al(0)) = (x_{0},\al_{0})$. Suppose that for some $\varrho > 0$, we have
\begin{displaymath}
\limsup_{t\to \infty} \frac1 t \log \E[|X(t)|^{p}] \le -\varrho < 0.
\end{displaymath} Then for any $\varrho > \e > 0$, there exists a positive constant $T$ such that $\E[|X(t)|^{p}] \le e^{-(\varrho -\e) t}$   for all $ t \ge T. $ This, together with Lemma 3.1 of \cite{ZhuYB-15}, implies that there exists some positive number $M$ so that \begin{equation}
\label{eq-p-moment-bdd}
\E[|X(t)|^{p}] \le M e^{-(\varrho -\e) t}\; \text{ for all } t \ge 0.
\end{equation}
Let $\delta > 0$. Then we have for any $k \in \mathbb N$,
\begin{equation}
\label{eq0-moment-as}
\begin{aligned}
 \E \biggl[ \supkd |X(t)|^{p}\biggr] &  \le  4^{p}\E \biggl[ |X(\kdel)|^{p} + \supkd \biggl| \int_{\kdel}^{t} b(X(s),\al(s)) \d s\biggr|^{p} \\
     &\qquad + \supkd \biggl| \int_{\kdel}^{t} \sigma(X(s),\al(s)) \d W(s)\biggr|^{p} \\
     & \qquad + \supkd \biggl| \int_{\kdel}^{t}\int_{\R^{n}_{0}} \gamma(X(s-),\al(s-),z) \tilde N(\d s, \d z) \biggr|^{p} \biggr].
\end{aligned}\end{equation}
Using \eqref{0-equilibrium} and \eqref{ito-condition}, we have
\begin{equation}
\label{eq1-moment-as}\begin{aligned}
\E\biggl[ \supkd \biggl| \int_{\kdel}^{t} b(X(s),\al(s)) \d s\biggr|^{p}\biggr] &  \le \E \biggl[\biggl| \int_{\kdel}^{k \delta} |b(X(s),\al(s))| \d s\biggr|^{p} \biggr] \\ & \le (\sqrt \kappa \delta)^{p} \E\biggl[ \supkd |X(t)|^{p}\biggr],
\end{aligned}\end{equation} where $\kappa > 0$ is the constant appearing in \eqref{ito-condition}.
On the other hand, by the Burkhoder-Davis-Gundy inequality (see, e.g., Theorem 2.13 on p. 70 of \cite{MaoY}) and \eqref{0-equilibrium}, \eqref{ito-condition}, we have
\begin{align}
\label{eq2-moment-as}
% \begin{aligned}
\nonumber \E    \biggl[ \supkd \biggl| \int_{\kdel}^{t} \sigma(X(s),\al(s)) \d W(s)\biggr|^{p} \biggr]   &  \le C_{p} \E\biggl[ \biggl( \int_{\kdel}^{k \delta} |\sigma(X(s),\al(s))|^{2} \d s \biggr)^{p/2} \biggr]    \\ & \le C_{p} (\kappa\delta)^{p/2} \E\biggl[ \supkd |X(t)|^{p}\biggr],
\end{align}
%\end{equation}
where  $C_{p}$ is a positive constant depending only on $p$.
Next we use Kunita's first inequality (see, e.g., Theorem 4.4.23 on p. 265 of \cite{APPLEBAUM}), %\eqref{0-equilibrium},
\eqref{ito-condition}, \eqref{eq2-ito-condition} and \eqref{eq-pth-moment-as-stable} to estimate
\begin{equation}
\label{eq3-moment-as}
\begin{aligned}
 \E &\biggl[  \supkd \biggl| \int_{\kdel}^{t}\int_{\R^{n}_{0}} \gamma(X(s-),\al(s-),z) \tilde N(\d s, \d z) \biggr|^{p} \biggr] \\
  & \le D_{p} \E\biggl[ \int_{\kdel}^{k\delta} \biggl(\int_{\R_{0}^{n}} |\gamma(X(s-),\al(s-),z)|^{2} \nu(\d z) \biggr)^{p/2} \d s\biggr]\\
  & \quad +  D_{p} \E\biggl[ \int_{\kdel}^{k\delta} \int_{\R_{0}^{n}} |\gamma(X(s-),\al(s-),z)|^{p} \nu(\d z)   \d s\biggr]\\
  & \le D_{p} [ \kappa^{p/2} + \hat \kappa]\delta \E\biggl[ \supkd |X(t)|^{p}\biggr],
\end{aligned}
\end{equation} where $\hat \kappa > 0$ is the constant appearing in \eqref{eq-pth-moment-as-stable} and $D_{p}$  is a positive constant depending only on $p$.

Now we plug \eqref{eq1-moment-as}, \eqref{eq2-moment-as}, and \eqref{eq3-moment-as} into \eqref{eq0-moment-as} to derive
\begin{displaymath}\begin{aligned}
&  \E  \biggl[ \supkd |X(t)|^{p}\biggr] \\ &\   \le 4^{p}\E [ |X(\kdel)|^{p}] + 4^{p}\big((\sqrt \kappa \delta)^{p} +  C_{p} (\kappa\delta)^{p/2}+  D_{p} [\kappa^{p/2} + \hat \kappa]\delta \big) \E\biggl[ \supkd |X(t)|^{p}\biggr].
\end{aligned}\end{displaymath} Now we choose a $\delta > 0$ sufficiently small so that $$4^{p}\big((\sqrt \kappa \delta)^{p} +  C_{p} (\kappa\delta)^{p/2}+  D_{p} [ \kappa^{p/2} + \hat \kappa]\delta \big) < \frac{1}{2}.$$  Then % for all $k \in \mathbb N$ sufficiently large so that $(k-1) \delta \ge T$,
 it follows from \eqref{eq-p-moment-bdd} that
\begin{equation}\label{eq-sup X(t)p mean}
 \E\biggl[ \supkd |X(t)|^{p}\biggr] \le 2M  4^{p}   e^{-(\varrho -\e) (k-1)\delta}.
\end{equation}
%Then following 
The rest of the proof uses the same arguments as those in the proof of Theorem 5.9 of \cite{MaoY}. For completeness, we include the details here. Thanks to \eqref{eq-sup X(t)p mean}, we have from the Chebyshev inequality that 
\begin{displaymath}
\P\biggl\{\omega\in \Omega: \supkd |X(t)| > e^{-(\varrho-2\e)(k-1)\delta/p} \biggr\} \le 2M 4^{p} e^{-\e (k-1)\delta}.
\end{displaymath} Then by the Borel-Cantelli lemma, there exists an $\Omega_{0}\subset \Omega$ with $\P(\Omega_{0}) =1$ such that for all $\omega\in \Omega_{0}$, there exists a $k_{0} = k_{0}(\omega) \in \N$ such that  \begin{displaymath}
\supkd |X(t, \omega)|  \le  e^{-(\varrho-2\e)(k-1)\delta/p},\quad \text{ for all } k \ge k_{0}=  k_{0}(\omega).
\end{displaymath} Consequently for all $\omega\in \Omega_{0}$,  if $(k-1 )\delta \le  t \le k\delta$ and $k\ge k_{0}(\omega)$, we have 
\begin{displaymath}
\frac1t \log(|X(t,\omega)|) \le -\frac{(\varrho-2\e)(k-1)\delta}{pt} \le -\frac{(\varrho-2\e)(k-1) }{pk}. 
\end{displaymath} This implies that \begin{displaymath}
\limsup_{t\to\infty}\frac1t \log(|X(t,\omega)|) \le -\frac{\varrho-2\e}{p},  \quad\text{ for all } \omega\in \Omega_{0}.
\end{displaymath}Now letting $\e \downarrow 0$, we obtain that $\limsup_{t\to\infty} \frac{1}{t} \log (|X(t)|) \le -\frac{\varrho}{p}$ a.s. This completes the proof.
\end{proof}

  \begin{proof}[Proof of Proposition \ref{prop-exp-stab-creterion-b}]  We need to find an appropriate Lyapunov function $V$ so that all conditions of Theorem \ref{thm-as-exp-stable} are satisfied. In addition, since $\al$ is an irreducible  continuous-time Markov chain with a stationary distribution $\pi= (\pi_{i},i\in \M)$, the assertion on a.s. exponential stability under \eqref{eq-exp-stab-creterion}  follows from Corollary \ref{coro-as-exp-stable}. To this end, let $G\in \S^{n\times n} $ and $p>0$ be as in the statement of the theorem. 
%  For each $i\in \M$,
We consider the Lyapunov function
  $$ V(x,i) =
(h_i-p\zeta_i)(x'G^2 x)\sp {p /2}, \ \ (x,i)\in \R^{n}\times \M ,$$
where $h_i>0$ such that $h_i-p\xi_i>0.$ % Clearly, for each $i\in \M$, $V(\cdot, i) $ is 
Let us now verify that $V$ satisfies conditions (i)--(v) of Theorem \ref{thm-as-exp-stable}.

  %where $0<p  <1$ is sufficiently small so that $1-p  \zeta_i >0$ for each $i\in \M$.
  It is readily seen that for each $i\in\M$,
$V(\cdot,i)$ is continuous, nonnegative, and vanishes only at $x=0$. Also observe that condition (i) of Theorem \ref{thm-as-exp-stable} is  satisfied with  $c_{1}(i): = (h_i-p\zeta_i)( \lambda_{\min}(G^{2}))^{\frac{p}{2}}.$ We can verify for $x \neq 0$ that
\begin{align*}
& DV(x,i) =(h_i-p\zeta_i)  p   (x'G^2 x) \sp {p /2 -1} G^{2} x,  \\
& D^{2}V(x,i) = (h_i-p\zeta_i) p   (x'G^2 x) \sp {p /2 -2}[(p  -2)  G^{2} x x' G^{2} + x'G^{2} x G^{2}].
\end{align*}
Then we compute \begin{align}
  \label{eq-L-V-computation}  \lefteqn{ \frac{1}{h_i-p\zeta_i} \op V(x,i) }\\
 \notag     &  = p  (x'G^2 x) \sp {\frac{p}{2} -1}   x' G_i^{2}  A_{i}x     + \frac{1}{2} \tr \big(  p   (x'G^2 x) \sp {\frac{p}{2} -2}[(p  -2)  G^{2} x x' G^{2} + x'G^{2} x G^{2}] B_{i}x x' B_{i}'\big)  \\
   \notag   &  \ \   +\int_{\R_{0}^{n}} \Big[ (x' (I + C_{i}(z) )' G^{2} (I + C_{i}(z)) x)^{\frac{p}{2}} - ( x' G^{2} x)^{\frac{p}{2}}  - p  ( x' G^{2} x)^{\frac{p}{2} -1} x'G^{2}   C_{i}(z) x\Big] \nu(\d z) \\
   \notag   & \ \ + \sum_{j\in \M} q_{ij} \frac{h_j-h_i-p  \zeta_{j}+p\zeta_i}{h_i-p\zeta_i}  (x'G^2 x) \sp {\frac{p}{2} } \\
 \notag   & =  p  (x'G^2 x) \sp {p /2 } \biggl[ \frac{x'  (G^{2}  A_{i} + A_{i}' G^{2} + B_{i}' G^{2} B_{i}) x }{2 x'G^2 x} + (p  -2) \frac{ (x' B_{i}' G^{2} x)^{2}}{2(x'G^2 x)^{2}}   + \frac{1}{p(h_i-p\zeta_i)}\sum_{j\in \M} q_{ij} h_{j}\\
 %   \notag   & \  \    + \frac{1}{p(h_i-p\zeta_i)}\sum_{j\in \M} q_{ij} h_{j}- \frac{1}{h_i-p\zeta_i}\sum_{j\in \M} q_{ij}\zeta_{j} \\
  \notag    & \  \    - \frac{1}{h_i-p\zeta_i}\sum_{j\in \M} q_{ij}\zeta_{j} +\int_{\R_{0}^{n}}  \biggl[ \frac{(x' (I + C_{i}(z) )' G^{2} (I + C_{i}(z)) x)^{\frac{p}{2}}}{ p  (x'G^2 x) \sp {\frac{p}{2} }} - \frac{1}{p } -  \frac{ x'G^{2}   C_{i}(z) x}{x'G^2 x}\biggr]\nu(\d z) \biggr].
\end{align}
Note that \begin{align}
\nonumber \frac{x'  (G^{2}  A_{i} + A_{i}' G^{2} + B_{i}' G^{2} B_{i}) x }{2 x'G^2 x} & = \frac{x' G (G A_{i} G^{-1} + G^{-1} A_{i}' G + G^{-1} B_{i}' G^{2} B_{i} G^{-1}) Gx}{2 | Gx|^{2}} \\
\label{eq1-L-V-computation}  & \le \frac{1}{2}\lambda_{\max}(G A_{i} G^{-1} + G^{-1} A_{i}' G + G^{-1} B_{i}' G^{2} B_{i} G^{-1})= \mu_{i}.
\end{align}
In addition, we have
\begin{equation}\label{eq2-L-V-computation}
\begin{aligned}
\frac{p-2}{2} \biggl(\frac{ x' B_{i}' G^{2} x}{x'G^2 x}\biggr)^{2} & =\frac{p-2}{8} \biggl(\frac{ x'G (G^{-1} B_{i}' G + G B_{i} G^{-1})G x}{x'G^2 x}\biggr)^{2}
 \\   &  \leq  \begin{cases}
\frac{p-2}{8}  \widehat\lambda^2(G^{-1} B_{i}' G+G B_{i} G^{-1} ),     & \text{ if } 0 < p \le 2, \\
 \frac{p-2}{8}   \rho^{2} (G^{-1} B_{i}' G+G B_{i} G^{-1} ),   & \text{ if }p > 2.
\end{cases} \\
& =  \frac{p-2}{8}  \Lambda^{2} (G^{-1} B_{i}' G+G B_{i} G^{-1} ) .
\end{aligned}\end{equation}
% where for a symmetric matrix $A$, we denote \[\widehat{\lambda}(A):=\left\{
%  \begin{array}{ll}
%   \lambda_{\max}(A), & \hbox{if}\;\;\lambda_{\max}(A)<0, \\
% 0\vee\lambda_{\min}(A), & \hbox{otherwise.}
%  \end{array} \right. \]

 On the other hand, since
% \begin{align*}
% \frac{ (x' B_{i}' G^{2} x)^{2}}{(x'G^2 x)^{2}} & = \biggl ( \frac{x' B_{i}' G^{2} x}{x'G^2 x}\biggr)^{2}=  \biggl ( \frac{x'G (G^{-1} B_{i}' G + G B_{i} G^{-1}) G x}{2 |Gx|^2  }\biggr)^{2}\\  &
%\end{align*}
\begin{align*}
 \frac{(x' (I + C_{i}(z) )' G^{2} (I + C_{i}(z)) x)^{\frac{p}{2}}}{  (x'G^2 x) \sp {\frac{p}{2} }} &= \biggl( \frac{x' (I + C_{i}(z) )' G^{2} (I + C_{i}(z)) x}{x'G^2 x} \biggr)^{\frac{p}{2}} \\ & \le   \bigl (  {\lambda_{\max}( G^{-1} (I + C_{i}(z) )' G^{2} (I + C_{i}(z))  G^{-1})} \bigr)^{\frac{p}{2}},\end{align*}
 and \begin{displaymath}
\frac{ x'G^{2}   C_{i}(z) x}{x'G^2 x} = \frac{x' G (G C_{i}(z) G^{-1} + G^{-1} C_{i}'(z) G) Gx}{2 |Gx|^{2}} \ge  \frac{1}   {2} \lambda_{\min} (G C_{i}(z) G^{-1} + G^{-1} C_{i}'(z) G),
\end{displaymath}
 it follows that  \begin{align}
 \nonumber & \int_{\R_{0}^{n}}    \biggl[ \frac{(x' (I + C_{i}(z) )' G^{2} (I + C_{i}(z)) x)^{\frac{p}{2}}}{ p  (x'G^2 x) \sp {\frac{p}{2} }} - \frac{1}{p } -  \frac{ x'G^{2}   C_{i}(z) x}{x'G^2 x}\biggr]\nu(\d z) \\
\nonumber   &\ \le  \frac{1}{p } \int_{\R_{0}^{n}}   \biggl[  \bigl (  {\lambda_{\max}( G^{-1} (I + C_{i}(z) )' G^{2} (I + C_{i}(z)) G^{-1}}\bigr)\bigr)^{\frac{p }{2}}   \\
 \nonumber      & \qquad \qquad \qquad
-  1 -  \frac{p}{2}  \lambda_{\min} (G C_{i}(z) G^{-1} + G^{-1} C_{i}'(z) G)\biggr] \nu(\d z)\\
\label{eq3-L-V-computation}   & \ = \frac{\rho_{i}}{p }.
\end{align}
Then upon putting the estimates \eqref{eq1-L-V-computation}--\eqref{eq3-L-V-computation} into \eqref{eq-L-V-computation}, we have \begin{align*}
  \op V(x,i)  &    \le  (h_i-p  \zeta_{i}) (x'G^2 x) \sp {\frac{p}{2} }\cdot\biggl[ p\mu_i
  +\frac{p-2}8{\Lambda}^2(GB_iG^{-1}+G^{-1}B_i'G)\\ 
  & \qquad  \qquad \qquad\qquad \qquad \qquad  + \frac{1}{h_i-p\zeta_i}\sum_{j\in \M} q_{ij} h_{j}- \frac{p}{h_i-p\zeta_i}\sum_{j\in \M} q_{ij}\zeta_{j}+ \eta_{i}\biggr]
 \\
    & = V(x,i) \biggl[  p\mu_i+\frac{p-2}8 {\Lambda}^2(GB_iG^{-1}+G^{-1}B_i'G) \\
    & \qquad  \qquad \qquad \qquad  \qquad \qquad  + \frac{1}{(h_i-p\zeta_i)}\sum_{j\in \M} q_{ij} h_{j} - \frac{p}{h_i-p\zeta_{i}} (\mu_{i}+\beta) + \eta_{i} \biggr],
    \end{align*}
    where we used \eqref{eq-com-0} to derive the last step. Thus condition (ii) of Theorem \ref{thm-as-exp-stable} is satisfied with $c_{2}(i) =  p\mu_i +\frac{p-2}8{\Lambda}^2(GB_iG^{-1}+G^{-1}B_i'G)- \frac{p}{h_i-p\zeta_{i}} (\mu_{i}+\beta)+ \frac{1}{(h_i-p\zeta_i)}\sum_{j\in \M} q_{ij} h_{j}+ \eta_{i}$.
In view of
\[
\aligned
|\langle DV(x,i),B_ix\rangle|^2= & \ p^2V^2(x,i)\biggl(\frac{x'G^2B_ix}{x'G^2x}\biggr)^2= \frac{p^{2}}{4}V^2(x,i)\biggl(\frac{x'G(GB_i G^{-1}+ G^{-1}B_{i}' G ) Gx}{x'G^2x}\biggr)^2 \\
\geq &\  \frac{ p^2}4\widehat{\lambda}^2(GB_iG^{-1}+G^{-1}B_i'G)V^2(x,i),
\endaligned
\]

  Note that
    \begin{align*}
&\int_{\R_{0}^{n}}  \biggl[\log\left(\frac{V(x+C_i(z)x,i)}{V(x,i)}\right)-\frac{V(x+C_i(z)x,i)}{V(x,i)}+1\biggr]\nu(\d z)\\
 & \  = \int_{\R_{0}^{n}}  \biggl[\frac{p}{2}\log\frac{x' (I+ C_{i}(z))' G^{2} (I + C_{i}(z))x}{ x' G^{2}x }-\frac{(x' (I+ C_{i}(z))' G^{2} (I + C_{i}(z))x)^{\frac{p}{2}}}{(x' G^{2}x)^{\frac{p}{2}} }+1\biggr]\nu(\d z)\\
 & \ \le 0\wedge \int_{\R_{0}^{n}}   \biggl[\frac{p}{2}\log{\lambda_{\max}(G^{-1}(I+ C_{i}(z))' G^{2} (I + C_{i}(z))G^{-1})}  \\ & \qquad \qquad- \bigl(  \lambda_{\min}(G^{-1}(I+ C_{i}(z))' G^{2} (I + C_{i}(z))G^{-1})\bigr)^{\frac{p}{2}}+1\biggr]\nu(\d z)\\
 & \ = -c_{4}(i).
\end{align*}
%\footnote{In checking the condition, I found that in many cases $c_4(i)$ is great than 0, such as $\log(e)-1+1=1>0$. The reason may be that the inequality $\log(x)-x+1\le0$ for $x\geq0$ is the best estimation, which shows that the last inequality but one of above formula may be too large. So, that will be better to replace it by
%\[
%-c_4=0\wedge\int_{\R_{0}^{n}}\biggl[\frac{p}{2}\log{\lambda_{\max}(G^{-1}(I+ C_{i}(z))' G^{2} (I + C_{i}(z))G^{-1})}- \bigl(  \lambda_{\min}(G^{-1}(I+ C_{i}(z))' G^{2} (I + C_{i}(z))G^{-1})\bigr)^{\frac{p}{2}}+1\biggr]\nu(dz).
%\]}
This establishes Condition (iv). Likewise, we can verify condition  (v) as follows.
\begin{align*}
 &\sum_{j\in \M} q_{ij} \left(\log V(x,j)-\frac{V(x,j)}{V(x,i)}\right) \\
 & \ \ = \sum_{j\in \M} q_{ij} \left(\log (h_j- p \zeta_{j})  +\frac{p}{2}\log (x'G^{2}x) - \frac{h_j-p \zeta_{j}}{h_i- p \zeta_{i}}\right) \\
 & \ \ = \sum_{j\in \M} q_{ij} \left(\log (h_j- p \zeta_{j})    - \frac{h_j-p \zeta_{j}}{h_i- p \zeta_{i}}\right)  = - c_{5}(i).
\end{align*}

Thus we have verified all conditions of Theorem \ref{thm-as-exp-stable} and hence in view of   Corollary \ref{coro-as-exp-stable}, \eqref{eq-exp-stab-creterion} implies the desired conclusion.
    \end{proof}
    
    \begin{proof}[Proof of Proposition \ref{prop-moment-exp-stab-1d}] This proof is motivated by the proofs of Theorem 5.24 in \cite{MaoY} and Theorem 3.3 in
\cite{Zong-14}.  As in the proof of Proposition \ref{prop-as-exp-stable-1d}, let us assume $x(0) = x \neq 0$ and $\al(0) = i \in \M$.
Then by the It\^o formula, we have
 \begin{align*}& |x(t)|^{p} \\
   &  \  = |x|^{p} \exp\biggl\{\int_0^t\biggl[pa(\alpha(s))-\frac{p}{2}b^2(\alpha(s))
 +p\int_{\R_0}[\log|1+c(\alpha(s-),z)|-c(\alpha(s-),z)]\nu(\d z)\biggr]\d s
   \\ &  \qquad \qquad  \qquad
   +   \int_0^t p b(\alpha(s))\d W(s)  +  \int_0^t\int_{\R_0}p \log|1+c(\alpha(s-),z)|\widetilde{N}(\d s,\d z) \biggr\}\\
     & \ = |x|^{p}
        \exp\biggl\{\int_0^t f(\al(s)) \d s\biggr\} \mathcal E (t),       \end{align*}
     where
      % \begin{align*}
  %   &    f(\al(s))  = pa(\alpha(s))+ \frac{1}{2}p(p-1)b^2(\alpha(s))    + \int_{\R_{0}}  [|1+ c(\al(s-),z)|^{p} - p c(\al(s-),z) -1]\nu(\d z),\end{align*} and
  $ \mathcal E (t)   : =  \exp \{ g(t) \}$ with
   \begin{align*}
 g(t)  &   =
            \int_0^t p b(\alpha(s))\d W(s)  - \frac{1}{2} \int_{0}^{t} p^{2} b^{2}(\al(s))\d s                                  +  \int_0^t\int_{\R_0}p \log|1+c(\alpha(s-),z)|\widetilde{N}(\d s,\d z) \\
                                 &   \qquad\qquad -\int_{0}^{t}\int_{\R_{0}}[|1+ c(\al(s-),z)|^{p}-p \log|1+ c(\al(s-),z)| -1] \nu(\d z)\d s.
 \end{align*}

 For each $t \ge 0$, let $\G_{t}: =\sigma(\al(s): 0\le s \le t)$, $\G:=\bigvee_{t\ge 0} \G_{t}$, and $\mathcal H_{t}: = \sigma(W(s), N([0,s)\times A), 0\le s \le t, A \in \B(\R_{0}))$.  Denote $\mathcal D_{t}: = \G \bigvee \mathcal H_{t}$.  Let $\{ \tau_{k}, k =1,2, \dots\}$ denote the sequence of switching times for the continuous-time Markov chain $\al(\cdot)$; that is, we define  $\tau_{1}: = \inf\{t \ge 0: \al(t) \neq \al(0) \} $ and $\tau_{k+1} : = \inf\{t \ge \tau_{k}: \al(t) \neq \al(\tau_{k}) \}$ for $k =1, 2,\dots$  It is well-known that $\tau_{k}\to \infty$ a.s. as $k\to \infty$.
Write     $\tau_{0}: =0$.
Then we can compute
\begin{align*}
\E[ |x(t)|^{p}] & = \E \Biggl[|x|^{p}  \sum_{k=0}^{\infty}I_{ \{\tau_{k}\le t < \tau_{k+1}\}} \exp\set{\int_{0}^{t} f(\al(s)) \d s} \EE(t) \Biggr]\\
 & =  \sum_{k=0}^{\infty}  \E \biggl[ |x|^{p}  \E\biggl[  I_{ \{\tau_{k}\le t < \tau_{k+1}\}}\exp\set{\int_{0}^{t} f(\al(s)) \d s} \EE(t)  \Big| \mathcal D_{\tau_{k}}\biggr]\biggr]\\
 & = \sum_{k=0}^{\infty}  \E \biggl[  |x|^{p} I_{ \{\tau_{k}\le t < \tau_{k+1}\}}\exp\set{\int_{0}^{t} f(\al(s)) \d s} \EE(\tau_{k}) \E\bigl[ \exp\{ g(t) - g(\tau_{k})\} \big |  \mathcal D_{\tau_{k}}\bigr]  \biggr].
\end{align*}
Note that on the event $ \{\tau_{k}\le t < \tau_{k+1}\}$, we have
 \begin{align*}
 g(t) -g(\tau_{k}) &   =
            \int_{\tau_{k}}^t p b(\alpha(\tau_{k}))\d W(s)  - \frac{1}{2} \int_{\tau_{k}}^{t} p^{2} b^{2}(\al(\tau_{k}))\d s                     \\
                            & \qquad             +  \int_{\tau_{k}}^t\int_{\R_0}p \log|1+c(\al(\tau_{k}-),z)|\widetilde{N}(\d s,\d z) \\
                                 &   \qquad  -\int_{\tau_{k}}^{t}\int_{\R_{0}}[|1+ c(\al(\tau_{k}-),z)|^{p}-p \log|1+ c(\al(\tau_{k}-),z)| -1] \nu(\d z)\d s.
 \end{align*}
 Then it follows from  the definition of the $\sigma$-algebra $\mathcal D_{\tau_{ k}}$ and Corollary 5.2.2 of \cite{APPLEBAUM} that $ \E\bigl[ \exp\{ g(t) - g(\tau_{k})\}  |  \mathcal D_{\tau_{k}}\bigr]  =1$. Consequently, we have
 \begin{align*}
\E& [ |x(t)|^{p}] \\ & = \sum_{k=0}^{\infty}  \E \biggl[ |x|^{p} I_{ \{\tau_{k}\le t < \tau_{k+1}\}}\exp\set{\int_{0}^{t} f(\al(s)) \d s} \EE(\tau_{k}) \biggr] \\
 & =  \sum_{k=0}^{\infty}  \E \biggl[\E \biggl[ |x|^{p} I_{ \{\tau_{k}\le t < \tau_{k+1}\}}\exp\set{\int_{0}^{t} f(\al(s)) \d s} \EE(\tau_{k}) \Big | \mathcal D_{\tau_{k-1}} \biggr] \biggr] \\
 & =   \sum_{k=0}^{\infty}  \E \biggl[|x|^{p} I_{ \{\tau_{k}\le t < \tau_{k+1}\}}\exp\set{\int_{0}^{t} f(\al(s)) \d s} \EE(\tau_{k-1}) \E [\exp\{g(\tau_{k}) - g(\tau_{k-1}) |  \mathcal D_{\tau_{k-1}} \} ] \biggr].
\end{align*} As argued before, we have $ \E [\exp\{g(\tau_{k}) - g(\tau_{k-1}) |  \mathcal D_{\tau_{k-1}} \} ] =1$ and hence
\begin{displaymath}
\E[ |x(t)|^{p}] =   \sum_{k=0}^{\infty}  \E \biggl[ |x|^{p} I_{ \{\tau_{k}\le t < \tau_{k+1}\}}\exp\set{\int_{0}^{t} f(\al(s)) \d s} \EE(\tau_{k-1}) \biggr].
\end{displaymath}
Continue in this fashion and we derive that \begin{align*}
\E[ |x(t)|^{p}] & =   \sum_{k=0}^{\infty}  \E \biggl[ |x|^{p} I_{ \{\tau_{k}\le t < \tau_{k+1}\}}\exp\set{\int_{0}^{t} f(\al(s)) \d s} \EE(\tau_{0}) \biggr] \\
    & =  \sum_{k=0}^{\infty}  \E \biggl[|x|^{p}  I_{ \{\tau_{k}\le t < \tau_{k+1}\}}\exp\set{\int_{0}^{t} f(\al(s)) \d s}   \biggr] \\
    & =    \E \biggl[ |x|^{p}\exp\set{\int_{0}^{t} f(\al(s)) \d s}   \biggr].
\end{align*}
Then it follows \eqref{eq-Lambda(a)} that
\begin{align*}
 \lim_{t\to \infty} \frac{1}{t}\log(\E[ |x(t)|^{p}] ) & =   \lim_{t\to \infty} \frac{1}{t} \log(|x|^{p}) +  \lim_{t\to \infty} \frac{1}{t} \log\bigg(\E\bigg[\exp\bigg \{ \int_{0}^{t} f(\al(s)) \d s\bigg\}\bigg]\bigg)
  % \\ &
  = \Upsilon(f).
\end{align*}
This completes the proof.
\end{proof}

    \begin{proof}[Proof of Proposition \ref{prop-as-moment-stab-M-matrix}]
In view of  Theorem \ref{thm-moment-stab-m-matrix}, we only need to verify Assumption \ref{assump-m-matrix} for the positive definite matrix $G^{2}$. But as observed in the proof of Proposition \ref{prop-exp-stab-creterion-b},  we have \begin{align*}
 \lan G^{2}x, A_{i}x\ran   +\frac12\lan B_{i}x,G^{2}B_{i}x\ran \le\frac12 \lambda_{\max}(GA_{i}G^{-1} +G^{-1}A_{i}G + G^{-1}B_{i}' G^{2} B_{i}G^{-1})  \lan x, G^{2}x\ran,         
\end{align*} and \eqref{eq3-L-V-computation} shows that  \begin{align*}
 \nonumber & \int_{\R_{0}^{n}}    \biggl[ \frac{(x' (I + C_{i}(z) )' G^{2} (I + C_{i}(z)) x)^{\frac{p}{2}}}{    (x'G^2 x) \sp {\frac{p}{2} }} - 1 - p \frac{ x'G^{2}   C_{i}(z) x}{x'G^2 x}\biggr]\nu(\d z)  \le  \eta_{i}.
 \end{align*}Finally we observe that \begin{displaymath}
\frac{(\lan x, G^{2}B_{i}x\ran)^{2}}{(\lan x, G^{2} x\ran)^{2}} = \frac14 \biggl(\frac{x' G (G^{-1} B_{i}'G + GB_{i}G^{-1})G x}{x'G^{2}x} \biggr)^{2} \le \frac14 \rho^{2} (G^{-1} B_{i}'G + GB_{i}G^{-1}). 
\end{displaymath} Therefore we have verified \eqref{eq1-moment-stable-M}--\eqref{eq3-moment-stable-M} of Assumption \ref{assump-m-matrix}. Then the assertions of the proposition follows from Theorem \ref{thm-moment-stab-m-matrix} directly.   
\end{proof}
% \vspace{3mm}

\bibliographystyle{apalike}

\end{document}